\documentclass[11pt]{article}

\usepackage{amsmath, amsfonts, amsthm, amsopn}
\usepackage{hyperref, cleveref}
\usepackage{caption, subcaption, booktabs, float, multirow}
\usepackage{graphicx, epstopdf}
\usepackage{fullpage}
\usepackage{algorithm, algorithmic}
\usepackage{enumerate}
\usepackage{bm, mathrsfs}
\usepackage{MnSymbol}
\usepackage{lipsum}

\ifpdf
\DeclareGraphicsExtensions{.eps,.pdf,.png,.jpg}
\else
\DeclareGraphicsExtensions{.eps}
\fi

\newcommand{\dif}{\mathrm{d}}

\newcommand{\mr}{\mathbb{R}}
\newcommand{\mc}{\mathbb{C}}

\newcommand{\diag}{diag}
\newcommand{\vect}{vec}

\newcommand{\bff}{\mathbf{f}}
\newcommand{\bfF}{\mathbf{F}}
\newcommand{\vphi}{\varphi}
\newcommand{\tensor}[1]{\mathbf{#1}}
\newcommand{\vphin}{\varphi^N}
\newcommand{\bvphin}{\bm{{\varphi}}^N}

\newtheorem{theorem}{Theorem}[section]

\newtheorem{definition}[theorem]{Definition}

\newtheorem{lemma}[theorem]{Lemma}
\newtheorem{remark}[theorem]{Remark}

\crefname{hypothesis}{Hypothesis}{Hypotheses}

\title{The Adaptive Spectral Koopman Method for Dynamical Systems}

\author{
Bian Li\thanks{Department of Industrial and Systems Engineering, Lehigh University, Bethlehem, PA (\href{mailto:bil215@lehigh.edu}{bil215@lehigh.edu}, \href{mailto:xiy518@lehigh.edu}{xiy518@lehigh.edu})}
\and 
Yi-An Ma\thanks{Halicio\u{g}lu Data Science Institute \& Department of Computer Science and Engineering, University of California San Diego, San Diego, CA  
	(\href{mailto:yianma@ucsd.edu}{yianma@ucsd.edu})}
\and 
J. Nathan Kutz\thanks{Department of Applied Mathematics, University of Washington, Seattle, WA
	(\href{mailto:kutz@uw.edu}{kutz@uw.edu})}
\and 
Xiu Yang\footnotemark[1]}

\begin{document}

\maketitle

\begin{abstract}
  Dynamical systems have a wide range of applications in mechanics, electrical engineering, chemistry, and so on. 
  In this work, we propose the adaptive spectral Koopman (ASK) method to solve nonlinear autonomous dynamical systems. 
  This novel numerical method leverages the spectral-collocation (i.e., pseudo-spectral) method and properties of the Koopman operator to obtain the solution of a dynamical system.
  Specifically, this solution is represented as a linear combination of the multiplication of Koopman operator's eigenfunctions and eigenvalues, and these eigenpairs are approximated by the spectral method.
  Unlike conventional time evolution algorithms such as Euler’s scheme and the Runge-Kutta scheme, ASK is mesh-free, and hence is more flexible when evaluating the solution.
  Numerical experiments demonstrate high accuracy of ASK for solving one-, two- and three-dimensional dynamical systems.
\end{abstract}


\section{Introduction}

The Koopman operator, introduced in 1931 by B. O.
Koopman~\cite{koopman1931hamiltonian}, is an infinite-dimensional
\emph{linear} operator that describes 
the evolution of a set of observables rather than the system state itself. The Koopman operator approach to \emph{nonlinear} dynamical systems has attracted considerable attention recently, 
as it provides a rigorous method for globally linearizing the system dynamics.
Specifically, because it is a linear operator, one can define its eigenvalues, eigenfunctions, and modes, and use them to represent dynamically interpretable low-dimensional embeddings of high-dimensional state spaces, which helps to understand the behavior of the underlying system and construct solutions through linear superposition~\cite{brunton2021modern}. 
In this procedure, the system dynamics is typically decomposed into linearly independent Koopman modes even if the system is nonlinear.
In particular, as pointed out in~\cite{mezic2020spectrum,korda2020data,nakao2020spectral}, if the dynamics is ergodic but non-chaotic, the spectrum of the Koopman operator in properly defined spaces does not contain continuous spectra, and the observable of the system can be represented as a linear combination of eigenfunctions associated with discrete eigenvalues of the Koopman operator. 
   
The Koopman operator provides powerful analytic tools to understand behaviors of dynamical systems.  For example, dynamical evolution of a finite-dimensional system described by ordinary differential equations (ODEs) can be studied by conducting Koopman mode analysis. 
Such analysis starts with a choice of a set of linearly independent observables, and the Koopman operator is then analyzed through its action on the subspace spanned by the chosen observables~\cite{mezic2005spectral}.
Moreover, it is also shown that the Koopman operator approach can be formally generalized to infinite-dimensional dynamical systems described by partial differential equations (PDEs), providing new perspectives on the analysis and control of these nonlinear spatiotemporal dynamics~\cite{wilson2016isostable,nathan2018applied,page2018koopman,nakao2020spectral}.
In addition, ergodic quotients and eigenquotients allow the Koopman operator to be used for the extraction and analysis of \emph{invariant} and \emph{periodic} structures in the state space~\cite{budivsic2009approximate}.
Moreover, Mezi\'{c} provided a Hilbert space setting
for spectral analysis of \emph{disspative dynamical systems}, and proved that the spectrum of the Koopman operator on these spaces is the closure of the product of the ``on-attractor'' and ``off-attractor'' spectra~\cite{mezic2020spectrum}. 

On the computational side, most existing numerical schemes motivated by the Koopman operator are categorized as data-driven methods, as they use spatiotemporal data to approximate a few of the leading Koopman eigenvalues, eigenfunctions, and modes.
In particular, the emerging computational method {\em dynamics mode decomposition}
(DMD)~\cite{rowley2009spectral, schmid2010dynamic,
  tu2014dynamic,proctor2016dynamic, kutz2016multiresolution, nathan2018applied,askham2018variable} as well as its variant such as extended DMD (EDMD)~\cite{williams2015data} uses snapshots of a dynamical system to extract temporal features as well as correlated spatial activity via matrix decomposition techniques. 
DMD and EDMD produce results for any appropriately formatted set of data, but connecting these outputs to the Koopman operator requires additional knowledge about the nature of the underlying system in that the system should be autonomous.
Later, a modified EDMD~\cite{williams2016extending} was proposed to compensate for the effects of system actuation when it is used to explore state space during the data collection, reestablishing the connection between EDMD and the Koopman operator in this more general class of data sets.  A review of many of the DMD variants for approximating the Koopman operator can be found in Brunton et al~\cite{brunton2021modern}. Moreover, theoretical results of identifying Koopman eigenfrequencies and eigenfunctions from a discretely sampled time series generated by such a system with unknown dynamics is provided in~\cite{das2020koopman} for a Fourier function.

Our aim in this paper is to provide a numerical method based on the
spectral-collocation method (i.e., the pseudospectral method) to implement the Koopman-operator approach to solving nonlinear ordinary differential equations (ODEs). 
Unlike the data-driven methods, this approach is on the other end of the ``spectrum'' of numerical methods, as it is based on the classical spectral method~\cite{fornberg1998practical,trefethen2000spectral}.
The main idea is to approximate eigenvalues, eigenfunctions, and modes of the Koopman operator based on its discretized form. 
Specifically, this method uses the differentiation matrix in the spectral method to approximate the generator of the Koopman operator, and then conducts eigendecomposition numerically to obtain eigenvalues and eigenvectors that approximate Koopman operator's eigenvalues and eigenfunctions, respectively.
Here, each element of an eigenvector is the approximation of the associated eigenfunction evaluated at a collocation point.
The modes are approximated by the computed eigenvalues, eigenvectors, and the initial state (or observable).
This work focuses on autonomous systems, and it would serve as a starting point for a new framework of numerical methods for dynamical systems.

The paper is organized as follows. Background topics are introduced in \Cref{sec:background}. 
Then, the adaptive spectral Koopman method is discussed in detail in \Cref{sec:ASK}. We present numerical results in \Cref{sec:experiments}, and the discussion and conclusions follow in \Cref{sec:dicussion_conclusion}.


\section{Background}
\label{sec:background}

\subsection{Koopman operator} 
\label{subsec:koopman}
Borrowing notions from~\cite{kutz2016dynamic}, we consider an autonomous system described by the ordinary differential equations
\begin{equation}
  \label{eq:auto}
  \frac{\dif \bm x}{\dif t} = \bff(\bm x),
\end{equation}
where the state $\bm x=(x_1,x_2,\dotsc,x_d)^\top$ belongs to a $d$-dimensional smooth manifold $\mathscr{M}$, and the dynamics $\bff: \mathscr{M} \rightarrow \mathscr{M}$ does not explicitly depend on time $t$. 
Here, $\bff$ is a possibly nonlinear vector-valued smooth function, of the same dimension as $\bm x$.
In many studies, we are concerned with the behavior of observables on the state space.
To this end, we define an observable to be a scalar function
$g:\mathscr{M}\rightarrow\mathbb{R}$, where $g$ is an element
of some function space $\mathcal{G}$
(e.g., $\mathcal{G}=L^2(\mathscr{M})$ as in~\cite{mezic2005spectral}). 
The flow map $\bfF_t: \mathscr{M} \rightarrow \mathscr{M}$ induced by the dynamical system~\eqref{eq:auto} depicts the evolution of the system as
\begin{equation}
\label{eqn:flow}
    \bm x(t_0 + t) = \bfF_t(\bm x(t_0)) = \bm x(t_0) + \int_{t_0}^{t_0 + t} \bff(\bm x(s)) \, \dif s. 
\end{equation}
Now we define the Koopman operator for continuous-time dynamical systems as follows~\cite{mezic2020spectrum}:
\begin{definition}
  Consider a family of operators $\{\mathcal{K}_t\}_{t\geq 0}$ acting on the space of observables so that
  \[\mathcal{K}_t g(\bm x_0) = g(\mathbf{F}_t(\bm x_0)),\] where $\bm x_0 = \bm x(t_0)$. 
  We call the family of operators $\mathcal{K}_t$ indexed by time t the Koopman
  operators of the continuous-time system~\eqref{eq:auto}.
\end{definition}
By definition, $\mathcal{K}_t$ is a \emph{linear} operator acting on the
function space $\mathcal{G}$ for each fixed $t$. Moreover, $\{\mathcal{K}_t\}$
form a semi-group.



\subsection{Infinitesimal generator} 
\label{subsec:generator}
The Koopman spectral theory~\cite{mezic2005spectral, rowley2009spectral} 
reveals properties that enable the Koopman operator to convert nonlinear finite-dimensional dynamics into linear infinite-dimensional dynamics. 
A key component in such spectral analysis is the infinitesimal generator (or generator for brevity) of the Koopman operator.
Specifically, 
the generator of the Koopman operator $\mathcal{K}_t$, denoted as $\mathcal{K}$,
is given by
\begin{align}
  \mathcal{K} g = \lim_{t \rightarrow 0} \frac{\mathcal{K}_t g
  -g}{t}. \label{eqn:Koopman_def_generator}
\end{align}
For any smooth function $g$, \Cref{eqn:Koopman_def_generator} implies that
\begin{align}
  \label{eqn:Koopman_derivative_relation}
  \mathcal{K}g(\bm x) = \frac{\dif g(\bm x)}{\dif t} = \nabla g(\bm x) \cdot \frac{\dif \bm x}{\dif t}. 
\end{align}
Denoting $\vphi$ an eigenfunction of $\mathcal{K}$ and $\lambda$ the eigenvalue associated with $\vphi$, we have
\begin{align}
  \mathcal{K} \vphi(\bm x) = \lambda \vphi(\bm x). 
    \label{eqn:Koopman_eigen}
\end{align}
Thus,
\begin{align}
    \lambda \vphi(\bm x) = \mathcal{K} \vphi(\bm x) = \frac{\dif \vphi(\bm x)}{\dif t}. 
    \label{eqn:Koopman_linear_relation}
\end{align}
This implies that $\varphi(\bm x(t_0+t))=e^{\lambda t}\varphi(\bm x(t_0))$, i.e., 
\begin{equation}
  \mathcal{K}_t \varphi(\bm x(t_0)) = e^{\lambda t} \varphi(\bm x(t_0)).
\end{equation}
Therefore, $\varphi$ is an eigenfunction of $\mathcal{K}_t$ associated with
eigenvalue $\lambda$. Of note, following the conventional notation, the eigenpair for 
$\mathcal{K}_t$ is considered as $(\varphi, \lambda)$ instead of $(\varphi,
e^{\lambda t})$.

Now suppose $g$ exists in the function space spanned by all the eigenfunctions
$\vphi_j$ (associated with eigenvalues $\lambda_j$) of $\mathcal{K}$, i.e.,
$g(\bm x) = \sum_j c_j \vphi_j(\bm x)$, then
\begin{align}
  \mathcal{K}_t [g(\bm x(t_0))] = \mathcal{K}_t \left[ \sum_j c_j
  \vphi_j(\bm x(t_0)) \right]
  = \sum_{j} c_j \mathcal{K}_t[\vphi_j(\bm x(t_0))] 
    \label{eqn:koopman_solution}.
\end{align}
Hence, 
\begin{equation}
  \label{eq:inf_expansion}
  g(\bm x(t_0+t)) = \sum_{j} c_j \vphi_j(\bm x(t_0)) e^{\lambda_j t}. 
\end{equation}
Similarly, if we choose a vector-valued observable $\bm g: \mathscr{M} \rightarrow \mr^d$ with $\bm g:=(g_1(\bm x), g_2(\bm x), \dotsc, g_d(\bm x))^\top$, the system of observables becomes 
\begin{align}
    \frac{\dif \bm g(\bm x)}{\dif t} = \mathcal{K} \bm g(\bm x)
    = 
    \begin{bmatrix}
        \mathcal{K} g_1(\bm x) \\
        \mathcal{K} g_2(\bm x) \\
        \vdots \\
        \mathcal{K} g_d(\bm x) \\
    \end{bmatrix}
    = \sum_j \lambda_j \vphi_j(\bm x) \bm c_j, \label{eqn:Koopman_system_observable}
\end{align}
where $\bm c_j \in \mc^d$ is called the $j$th Koopman mode with
$\bm c_j := (c_j^1, c_j^2, \dotsc, c_j^d)^\top$. 
In general, there is no universal guide for choosing observables as this choice is problem dependent.  
A good set of observables can lead to a system that is significantly easier to
solve. An example from~\cite{brunton2016koopman,lusch2018deep} is illustrated in
~\Cref{append:obser}.

We finalize the introduction of the Koopman operator with the following simple example. 
Consider the system $\frac{\dif x}{\dif t}=\mu x$ with $x, \mu\in\mathbb{R}$ and $\mu\neq 0$. 
Then, one can easily verify that $\varphi_n(x):=x^n$ is an eigenfunction of the Koopman operator associated with this dynamical system, and the corresponding eigenvalue is $\lambda_n= n \mu$ with $n\in\mathbb{N}^+$ (a similar example is presented in~\cite{budivsic2012applied}).
According to~\Cref{eq:inf_expansion}, by setting $g(x)=x$ and let $x(0)=x_0$, we have
\[  x(t) = \sum_{j=1}^\infty c_j \vphi_j(x_0) e^{\lambda_j t} =\sum_{j=1}^\infty c_j x_0^j e^{\mu j t}. \]
Setting $t=0$ gives $x_0=x(0) = \sum_{j=1}^\infty c_j x_0^j$, which indicates $c_1=1$ and $c_j=0$ when $j\neq 1$. 
Therefore, we obtain the solution of the ODE as $x(t)=x_0e^{\mu t}$. 


\section{Adaptive Spectral Koopman Method}
\label{sec:ASK}

In this section, we introduce the adaptive spectral Koopman (ASK) method, which is a numerical method based on the Koopman operator and the spectral method to solve ODE systems.
Before describing details of this method, we introduce the notations used in this algorithm. 
Let $\bm x(t)$ denote the solution of an ODE system with an initial condition $\bm x(t_0)=\bm x_0$. Assume $t_0=0$ in~\cref{eqn:flow}, we consider solutions in time interval $[0, T]$ with $T>0$.
Letter $n$ denotes the number of ``check points'' (see details in~\cref{subsec:adapt}). 
The radius of the neighborhood of $\bm x(t)$ is denoted by $r$ while $\gamma$ is a parameter that controls the update of the neighborhood. 

\subsection{Finite-dimensional approximation}
\label{subsec:finite_approximation}
Based on the preliminaries introduced in~\Cref{subsec:generator},
we aim to identify the following truncated approximation of~\Cref{eq:inf_expansion}
\begin{equation}
  \label{eq:trun_expansion}
  g(\bm x(t)) \approx g_N(\bm x(t)) = \sum_{j=0}^N \tilde c_j \vphin_j(\bm x_0)
  e^{\tilde \lambda_j t}. 
\end{equation}
where $\vphin_j$ are polynomial approximations of $\vphi_j$, $\tilde\lambda_j$ and
$\tilde c_j$ approximate $\lambda_j$ and $c_j$, respectively.
Next, because $\frac{\dif \bm x}{\dif t}=\bff(\bm x)$, \Cref{eqn:Koopman_derivative_relation} and \Cref{eqn:Koopman_linear_relation} indicate that for any eigenfunction $\varphi$,
\begin{align*}
    \mathcal{K} \vphi &= \bff \cdot \nabla \vphi 
    = \left( f_1 \frac{\partial \vphi}{\partial x_1} + f_2 \frac{\partial \vphi}{\partial x_2} + ... + f_d \frac{\partial \vphi}{\partial x_d} \right) \\
    &= \left( f_1 \frac{\partial}{\partial x_1} + f_2 \frac{\partial}{\partial x_2} + ... + f_d \frac{\partial}{\partial x_d} \right)(\vphi).
\end{align*}
Thus, 
\begin{align}
    \mathcal{K} = f_1 \frac{\partial}{\partial x_1} + f_2 \frac{\partial}{\partial x_2} + ... + f_d \frac{\partial}{\partial x_d}. \label{eqn:finite_approximation_fundamental}
\end{align}
Here, we consider the case with $d\leq 3$, and adopt the approaches in the spectral-collocation method.  Specifically, our algorithm uses Gauss-Lobatto points for the interpolation of $\varphi$ and approximates (partial) derivatives with differentiation matrices (see e.g.,~\cite{karniadakis2005spectral, tang2006spectral, hesthaven2007spectral}) in \Cref{eqn:finite_approximation_fundamental}.
Consequently, the first step is to discretize $\mathcal{K}$. 

\begin{enumerate}[(1)]
  \item When $d = 1$. Let $\{\xi_i\}_{i=0}^N$ be the Gauss-Lobatto
    points and the polynomial interpolation of $\varphi(\bm x)$ is 
    \[\varphi(x)\approx \vphin(x) := \sum_{i=0}^N \vphin(\xi_i) P_i(x),\]
    where the basis functions $P_j$ are Lagrange polynomials satisfying $P_j(\xi_i) = \delta_{ij}$ and $\delta_{ij}$ is the Kronecker delta function. 
    Namely, $\vphin(x)$ is the projection of $\varphi(x)$ on the space $\text{span}\{P_j(x)\}_{j=0}^N$. 
    Let $\bvphin=[\vphin(\xi_0), \vphin(\xi_1), \dotsc, \vphin(\xi_N)]^\top$, we have
\begin{align}
  \mathcal{K}\bvphin = \diag(\bff(\xi_0), \bff(\xi_1), \dotsc, \bff(\xi_N)) \tensor D\bvphin :=\tensor K\bvphin, 
  \label{eqn:finite_approximation_1D}
\end{align}
    where $\tensor D$ is the differentiation matrix associated with $\{\xi_i\}_{i=0}^N$ and $\tensor K$ is an $(N+1)\times (N+1)$ matrix. 
    Here, we abuse the notation to let
    $\mathcal{K}\bvphin=[\mathcal{K}\vphin(\xi_0), \mathcal{K}\vphin(\xi_1),
    \dotsc, \mathcal{K}\vphin(\xi_N)]^\top$,
    and similar notations are used in the following $d=2,3$ cases.

  \item When $d = 2$. Let $\{\xi_i\}_{i=0}^N$ and $\{\eta_j\}_{j=0}^N$ be 
 the Gauss-Lobatto points of $x_1$ and $x_2$, respectively.
    Every eigenfunction $\vphi$ is now a bivariate function, whose polynomial interpolation $\vphin$ is
    \[\varphi(x_1,x_2)\approx \vphin(x_1,x_2) := \sum_{i=0}^N \sum_{j=0}^N \vphin(\xi_i,\eta_j) P_i(x_1)P_j(x_2).\]
    Hence, we define a matrix $\tensor\Phi^N$ as
\begin{align*}
    \tensor\Phi^N = 
	\begin{bmatrix}
		\vphin(\xi_0, \eta_0) &\vphin (\xi_0, \eta_1) & \dots & \vphin(\xi_0, \eta_N) \\
		\vphin(\xi_1, \eta_0) & \vphin(\xi_1, \eta_1) & \dots & \vphin(\xi_1, \eta_N) \\
		\vdots			& \vdots		  & \ddots & \vdots \\
		\vphin(\xi_N, \eta_0) & \vphin(\xi_N, \eta_1) & \dots & \vphin(\xi_N, \eta_N) \\
	\end{bmatrix}.
\end{align*}
Let $\tensor D_1$ and $\tensor D_2$ be the differentiation matrices for $x_1$ and $x_2$, respectively, and $\tensor F_1$ and $\tensor F_2$ be the matrices of $f_1$ and $f_2$ evaluated at $(\xi_i, \eta_j)$.
Also, we denote $\mathcal{K}\tensor\Phi^N$ the matrix with elements $(\mathcal{K}\tensor\Phi^N)_{ij}=\mathcal{K}\tensor\Phi^N(\xi_i,\eta_j)$.
Then, $\mathcal{K}\tensor\Phi^N$ can be computed as
\begin{align*}
  \mathcal{K}\tensor\Phi^N = \tensor F_1 \odot \left(\tensor D_1\tensor\Phi^N
  \right) + \tensor F_2 \odot \left(\tensor\Phi^N \tensor D_2^\top \right), 
\end{align*}
where $\odot$ denotes the Hadamard product. 
In the computation, we vectorize $\tensor\Phi^N$ (along columns) to obtain
\begin{align*}
  \mathcal{K} \vect(\tensor\Phi^N) 
    &= \vect(\tensor F_1) \odot \Big( (\tensor I \otimes \tensor D_1) \vect(\tensor\Phi^N) \Big) 
        + \vect(\tensor F_2) \odot \Big( (\tensor D_2 \otimes \tensor I) \vect(\tensor\Phi^N) \Big) \\
    &= \Big[ \diag \big(\vect(\tensor F_1) \big) (\tensor I \otimes
    \tensor D_1) 
        +  \diag \big(\vect(\tensor F_2) \big) (\tensor D_2 \otimes \tensor I) \Big] \Big( \vect(\tensor\Phi^N) \Big) \\
  & := \tensor K\vect(\tensor\Phi^N),
\end{align*}
where $\otimes$ denotes the Kronecker product, $\tensor I$ is the identity matrix, and $\tensor K$ is an $(N+1)^2 \times (N+1)^2$ matrix.

  \item When $d = 3$. Let $\{\xi_i\}_{i=0}^N$, $\{\eta_j\}_{j=0}^N$, and $\{\zeta_k\}_{k=0}^N$ be the Gauss-Lobatto points of $x_1$, $x_2$, and $x_3$, respectively. 
    The collocation points are then $(\xi_i, \eta_j, \zeta_k)$. 
    In this case, $\vphi$ is approximated as
    \[\varphi(x_1,x_2, x_3)\approx \vphin(x_1,x_2,x_3) := \sum_{i=0}^N \sum_{j=0}^N \sum_{k=0}^N\vphin(\xi_i,\eta_j, \zeta_k) P_i(x_1)P_j(x_2)P_k(x_3).\]
    Hence, the values of $\vphin$ at the collocation points can be represented by a tensor $\tensor\Phi^N$ whose frontal slices are written as
\begin{align*}
	\tensor{\Phi}^N(:,:, k) = 
	\begin{bmatrix}
		\vphin(\xi_0, \eta_0, \zeta_k) &\vphin (\xi_0, \eta_1, \zeta_k) & \dots & \vphin(\xi_0, \eta_N, \zeta_k) \\
		\vphin(\xi_1, \eta_0, \zeta_k) & \vphin(\xi_1, \eta_1, \zeta_k) & \dots & \vphin(\xi_1, \eta_N, \zeta_k) \\
		\vdots			& \vdots		  & \ddots & \vdots \\
		\vphin(\xi_N, \eta_0, \zeta_k) & \vphin(\xi_N, \eta_1, \zeta_k) & \dots & \vphin(\xi_N, \eta_N, \zeta_k) 
	\end{bmatrix}.
\end{align*}
With the n-mode multiplication in tensor algebra, we arrive at a compact representation of the approximation,
\begin{align*}
  \mathcal{K}\tensor \Phi^N = \bfF_1 \odot \Big(\tensor \Phi^N \times_1
  \tensor D_1\Big) + \bfF_2 \odot \Big(\tensor \Phi^N \times_2 \tensor D_2\Big) +
  \bfF_3 \odot \Big(\tensor \Phi^N \times_3 \tensor D_3\Big), 
\end{align*}
where $\times_p$ denotes the mode-$p$ tensor-matrix multiplication. Here,
$\tensor D_1, \tensor D_2, \tensor D_3$ are the differentiation matrices, and $\bfF_1, \bfF_2, \bfF_3$ denote the tensors resulting from $f_1, f_2, f_3$ evaluated at $(\xi_i, \eta_j, \zeta_k)$. Following the same idea of vectorization, we rewrite the tensor representation as
\begin{align*}
  \mathcal{K} \vect(\tensor \Phi^N)
	&= \vect(\bfF_1) \odot \Big( (\tensor I \otimes \tensor I  \otimes
  \tensor D_1) \vect(\tensor \Phi^N) \Big) \\
	&+ \vect(\bfF_2) \odot \Big( (\tensor I \otimes \tensor D_2 \otimes \tensor I)
  \vect(\tensor \Phi^N) \Big)  \\  
	&+  \vect(\bfF_3) \odot \Big( (\tensor D_3 \otimes \tensor I \otimes \tensor
  I) \vect(\tensor \Phi^N) \Big) \\
	&= \Big[ \diag\big(\vect(\bfF_1)\big) (\tensor I \otimes \tensor I  \otimes \tensor  D_1) \\
	&+ \diag\big(\vect(\bfF_2)\big) (\tensor I \otimes \tensor D_2 \otimes
  \tensor I)  \\
	&+ \diag\big(\vect(\bfF_3)\big) (\tensor D_3 \otimes \tensor I \otimes
  \tensor I) \Big] \Big( \vect(\tensor \Phi^N) \Big) \\
  & := \tensor K\vect(\tensor \Phi^N) .
\end{align*}
where $\tensor K$ is an $(N+1)^3 \times (N+1)^3$ matrix.
\end{enumerate} 
In all these cases, the discretized generator $\mathcal{K}$ can be represented as a matrix $\tensor K$. 
For $d=2$ and $d=3$, the total number of eigenfunctions used in~\Cref{eq:trun_expansion} is $(N+1)^2$ and $(N+1)^3$, respectively, instead of $(N+1)$.
For brevity, $g_N$ is still used to denote the approximated observable for different $d$.
The derivation of higher dimensional systems amounts to further extensions of the three-dimensional case by the Kronecker product.


\subsection{Eigen-decomposition} 
\label{subsec:eigen_decomposition}
Now the eigenvalue problem of the Koopman operator in~\Cref{eqn:Koopman_eigen} is discretized as the eigenvalue problem of matrix $\tensor K$, i.e.,
$\tensor K \bm v = \tilde\lambda \bm v$, where $\tilde\lambda \in \mc$ and $\bm v$ is a
complex vector. 
The vector $\bm v$ is an approximation of $\mathcal{K}$'s eigenfunction $\varphi$ evaluated at the collocation points and $\tilde\lambda$ is the approximation of the associated eigenvalue of $\mathcal{K}$. 
The matrix form of the eigenvalue problem is 
\begin{equation}
  \label{eq:eigen}
  \tensor K\tensor V = \tensor V\tensor \Lambda,
\end{equation}
where $\tensor V$ consists of columns $\bm v_j$ and the diagonal elements of
$\tensor\Lambda$ are $\tilde\lambda_j$.
By construction, for $d=1$, $(\bm v_j)_i=\vphin_j(\xi_i)\approx
\varphi_j(\xi_i)$, and for $d=2$ or $3$, $\bm v_j=\vect(\tensor \Phi^N_j)$, where $\tensor\Phi^N_j$ approximates the values of eigenfunction $\varphi_j$ at the collocation points.
Of note, the collocation points in multi-dimensional cases are constructed by the tensor product of one-dimensional collocation points, but we have not specified how to obtain such points, the details of which are given in~\Cref{subsec:solution}. 
Also, we emphasize that these collocation points are related to $\bm x$ instead of $t$. 
In other words, ASK discretizes $\varphi(\bm x)$ in space instead of discretizing $\bm x(t)$ in time, which is different from conventional spectral methods for ODEs.


\subsection{Constructing the solution} 
\label{subsec:solution}
Let us first consider $d=1$. 
By the eigen-decomposition, one can access values of eigenfunctions at the Gauss-Lobatto points $\bm\Xi:=\{\xi_i\}_{i=0}^N$, where $\xi_0<\xi_1<\dotsc<\xi_N$.
Therefore, $\varphi(\bm x_0)$ can be approximated when $\xi_0\leq \bm x_0\leq \xi_N$.
To avoid polynomial interpolation, ASK uses an even number for $N$ and sets $\xi_{N/2}=\bm x_0$. 
Based on this setting, we consider a neighborhood of $\bm x_0$ with radius $r$, i.e., $[\bm x_0-r, \bm x_0+r]$, where $r$ is tunable. Gauss-Lobatto points are then generated such that $\bm x_0-r = \xi_0<\xi_1<\dotsc<\xi_{N/2}=\bm x_0< \dotsc <\xi_N=\bm x_0+r$.  Thus, $g_N$ is constructed as
\begin{equation}
  \label{eq:1d_sol}
  g_N(\bm x(t)) = \sum_{j=0}^N \tilde c_j\vphin_j(\bm x_0)e^{\tilde\lambda_j t} =
  \sum_{j=0}^N \tilde c_j\varphi^N_j(\xi_{N/2})e^{\tilde\lambda_j t} = 
  \sum_{j=0}^N \tilde c_j(\bm v_j)_{N/2}e^{\tilde\lambda_j t},
\end{equation}
where $\bm v_j$ are eigenvectors of matrix $\tensor K$ computed in~\Cref{subsec:finite_approximation}. 

To approximate Koopman modes $c_j$, we set $t=0$ in~\Cref{eq:1d_sol}, which yields 
\[g(\bm x_0)\approx g_N(\bm x_0) = \sum_{j=0}^N \tilde c_j\vphin_j(\bm x_0),\]
which holds for different initial state $\bm x_0$, e.g., 
\[g(\xi_i)\approx g_N(\xi_i) = \sum_{j=0}^N \tilde c_j\vphin_j(\xi_i), \quad i=0,\dotsc, N,\]
where $\xi_i$ are the aforementioned Gauss-Lobatto points. 
Thus, we can obtain $\tilde c_j$ by solving a linear system $\tensor V\bm c=g(\bm\Xi)$,
where $\tensor V$ is defined in~\Cref{eq:eigen}, $g(\bm\Xi) = (g(\xi_0), \dotsc,
g(\xi_N))^\top$ and $\bm c=(\tilde c_0, \dotsc, \tilde c_N)^\top$.
As an example, if $g(\bm x) := \bm x$, then $g(\bm\Xi) = (\xi_0, \dotsc, \xi_N)^\top$.

For $d=2$, we consider the neighborhood of $\bm x_0=(x_0^1, x_0^2)^\top$ as $[x_0^1-r, x_0^1+r]\times [x_0^2-r, x_0^2+r]$. 
Similarly, for $d=3$, the neighborhood is $[x_0^1-r, x_0^1+r]\times [x_0^2-r, x_0^2+r]\times [x_0^3-r, x_0^3+r]$, where $\bm x_0=(x_0^1, x_0^2, x_0^3)^\top$.
We then generate $(N+1)$ Gauss-Lobatto points in each direction and use the tensor product rule to construct multi-dimensional collocation points.
In practice, one can use standard Gauss-Lobatto points in the spectral method such as Legendre-Gauss-Lobatto and Chebyshev-Gauss-Lobatto points.
Now the set of all collocation points is $\tensor\Xi = \{(\xi_i,\eta_j)\}_{i,j=0}^N$ for $d=2$ and $\tensor\Xi = \{(\xi_i,\eta_j,\zeta_k)\}_{i,j=0}^N$ for $d=3$. 
Of note, the isotropic set up is applied here for demonstration purpose, i.e., we use a fixed $r$ in each direction and admit the same number of Gauss-Lobatto points in each dimension. However, this is not necessarily the optimal choice, and one can use different $r$ and different numbers of Gauss-Lobatto points in different directions.

Next, since we vectorize matrix (or tensor) $\tensor\Phi^N$ column by column (or slice by slice) as shown in~\Cref{subsec:finite_approximation}, $\varphi_j(\bm x_0)$ is again approximated by the ``middle'' element of vector $\vect(\tensor\Phi_j^N)$, which leads to
\begin{equation}
  \label{eq:nd_sol}
  \bm g_N(\bm x(t)) =
  \begin{cases}
    \displaystyle\sum_{j=0}^{(N+1)^2-1}\bm c_j \: (\bm v_j)_{[(N+1)^2-1]/2} \: e^{\tilde\lambda_j t}, & d=2; \\
    \displaystyle\sum_{j=0}^{(N+1)^3-1}\bm c_j \: (\bm v_j)_{[(N+1)^3-1]/2} \: e^{\tilde\lambda_j t}, & d=3. \\
  \end{cases}
\end{equation}
Here, each element of the modes $\bm c_j=(\tilde c_j^1, \cdots,\tilde c_j^d)^\top$ corresponds to a component of $\bm g$, and it is computed in the same manner as in the $d=1$ case. 
For example, for $d=2$, i.e., $\bm g(\bm x)=(g_1(\bm x), g_2(\bm x))^\top$
(correspondingly, $\bm g_N(\bm x)=(g_N^1(\bm x), g_N^2(\bm x))^\top$), we have
$g_N^1(\bm x(t))=\sum_{j=0}^{(N+1)^2-1}\tilde c^1_j\vect(\tensor \Phi_j^N)_{[(N+1)^2-1]/2}e^{\tilde\lambda_j t}$.
Consider matrix $g_1(\tensor\Xi)$ whose elements are $(g_1(\tensor\Xi))_{ij} = g_1(\xi_i,\eta_j)$.
The modes $\bm c^1=(\tilde c_0^1, \dotsc, \tilde c_N^1)^\top$ are obtained by solving a linear system $\tensor V \bm c^1=\vect(g_1(\tensor\Xi))$.
Similarly, we can compute the modes for $g_N^2(\bm x)$. 
In practice, our algorithm solves the linear system $\tensor V\tensor C = \bm g(\bm\Xi)$, where $\tensor C=(\bm c^1, \bm c^2)$ and $\bm g(\bm\Xi)=(\vect(g_1(\bm\Xi)), \vect(g_2(\bm \Xi)))$.
The modes for $d=3$ are computed in the same manner. 
In addition, a pseudocode is presented in~\Cref{append:code} to illustrate how the solution is constructed.


\subsection{Adaptivity}
\label{subsec:adapt}

Since we apply a finite-dimensional approximation of the Koopman operator and exploit the Lagrange interpolation to approximate the eigenfunctions, the accuracy of the solution may decay as time evolves, especially for highly nonlinear systems.
To further improve the accuracy, we propose an adaptive approach to update $\tensor V,\tensor \Lambda$ and $\bm c_j$.
The main idea is to identify the time to repeat the procedure
described in~\Cref{subsec:finite_approximation}--~\Cref{subsec:solution}.
To this end, we set check points $0<\tau_1<\tau_2<\dotsc<\tau_n<T$ to examine the ``validity'' of the neighborhood of $\bm x(\tau_k)$. 
Specifically, the component of $\bm x(\tau_k)=(x_1(\tau_k),\dotsc,x_d(\tau_k))^\top$ is acceptable if $x_i(\tau_k)\in R_i$ where
\begin{align}
  R_i := [L_i + \gamma r_i, U_i - \gamma r_i].
  \label{eqn:acceptable_range}
\end{align}
Here, $L_i$ and $U_i$ are the lower and upper bounds, $r_i$ is the radius in the $i$th direction, and $\gamma\in (0,1]$ is a tunable parameter.
Recall that the isotropic setup is used in this work, thus $r_i\equiv r$.
In the initial step, $L_i := x_0^i-r_i$ and $U_i := x_0^i+r_i$, i.e., $\gamma=1$.
In practice, one can fix $\gamma=1$ (or other real number in $(0,1]$) and tune $r_i$ only.
Hereby, we keep both $\gamma$ and $r_i$ for future extension to anisotropic design and more advanced adaptivity criterion.

If $x_i(\tau_k) \in R_i$ for all $i$, then $R_1\times \dotsc \times R_d$ is a valid neighborhood of $\bm x(\tau_k)$. 
Otherwise, we update all $L_i, U_i$ and reconstruct $\vphin_j, \tilde\lambda_j,
\tilde c_j$ to obtain $\bm x(t)~ (t>\tau_k)$ as follows:
\begin{enumerate}

  \item Set $L_i = x_i(\tau_k) - r_i, U_i = x_i(\tau_k) + r_i$,
    $1\leq i\leq d$. \label{itm:update_1}

  \item Generate Gauss-Lobatto points and the differentiation matrix in each interval  $[L_i^k, U_i^k]$. Repeat the procedure in~\Cref{subsec:finite_approximation} to compute matrix $\tensor K$.
    \label{itm:update_2}

  \item Repeat the eigendecomposition in~\Cref{subsec:eigen_decomposition} to update $\tensor V$ and $\tensor \Lambda$ in~\Cref{eq:eigen}. 
    \label{itm:update_3}

  \item Compute coefficient $\bm c_j$ as in~\Cref{subsec:solution} with the updated $\tensor V$. 
    \label{itm:update_4}  

  \item Construct solution $\bm x(t)$ by replacing $e^{\tilde\lambda_jt}$ with $e^{\tilde\lambda_j (t-\tau_k)}$ in~\Cref{eq:1d_sol} (or~\Cref{eq:nd_sol} for $d=2,3$).

\end{enumerate}
Note that the modification of constructing the solution in step 5 is necessary because when an update is performed, we need to set $t_0=\tau_k$ and $\bm x_0=\bm x(t_0)=\bm x(\tau_k)$.

The parameter $\gamma$ decides how often we update the neighborhood and reconstruct the solution. By construction, a larger $\gamma$ demands updating the eigendecomposition more frequently. 
The extreme case $\gamma = 1$ enforces the update at every check point.
In this work, we set $\tau_{k+1}-\tau_k\equiv \Delta\tau$. 
Notably, since the solution is discretized \emph{in space} instead of in time as in conventional ODE solvers, the check points are different from time grids $0<t_1<t_2<\dotsc$ in those solvers.
If we set $k=0$, then no update is made, which indicates that the solution $\bm x(t)$ only relies on the eigendecomposition based on $\bm x_0$ (see the example pseudocode in~\Cref{append:code}). 


\subsection{Properties of the algorithm} 
\label{subsection:property}

In this work $\bm x, \bff, \bm g$ are real-valued functions.
Now we show that the solutions obtained by ASK are real numbers, although $\tensor V, \tensor\Lambda, \bm c_j$ may contain complex values. 
We start with reiterating a well-known conclusion:
\begin{lemma}
  \label{lemma:conjugate_pairs}
    If a real matrix has complex eigenvalues, then they always occur in complex conjugate pairs. Furthermore, a complex conjugate pair of eigenvalues have a complex conjugate pair of associated eigenvectors.
\end{lemma}
\begin{proof}
    Suppose the matrix $\tensor K \in \mr^{n \times n}$ has an eigenpair $\bm v$ and $\lambda$ such that $\tensor K \bm v = \lambda \bm v$. Let $\bar{\cdot}$ operator denote the complex conjugate. Taking the complex conjugate of both sides of the equation, we have $\bar{\tensor K} \bar{\bm v} = \bar{\lambda} \bar{\bm v}$. However, $\tensor K = \bar{\tensor K}$ since $\tensor K$ has real entries. Thus, $\tensor K \bar{\bm v} = \bar{\lambda} \bar{\bm v}$. The claim follows.  
\end{proof}

Our main theorem is presented next:
\begin{theorem}\label{thm:real_solutions}
    ASK yields real-valued solutions for dynamical systems with real-valued $\bm x, \bff$ and $\bm g$. 
\end{theorem}
\begin{proof}
  We only need to consider the $d=1$ case since the solution for high-dimensional cases are constructed in the same manner.
  Let $\bm v$ be a eigenvector, then it is a column of matrix $\tensor V$ in~\Cref{eq:eigen}.
  It is only necessary to consider the case where $\bm v$ is not a real-valued vector.
  According to~\Cref{lemma:conjugate_pairs}, $\bar{\bm v}$ is also a column of $\tensor V$.
  Let $\bm u$ be a row of $\tensor V^{-1}$ such that $\bm u\cdot \bm v=1$ and $\bm u\cdot \tilde{\bm v}=0$, where $\tilde{\bm v}$ is any column of $\tensor V$ other than $\bm v$.
  It is clear that $\bar{\bm u}\cdot\bar{\bm v}=1$ and $\bar{\bm u}\cdot\tilde{\bm v}=0$.
  Therefore, $\bar{\bm u}$ is also a row of $\tensor V^{-1}$.
  Next, as shown in~\Cref{subsec:solution}, we compute the modes $\bm c$ as $\bm c = \tensor V^{-1} \bm g(\tensor \Xi)$.
    Let $c_m$ be the element of $\bm c$ such that $c_m=\bm u\cdot \bm g(\tensor
    \Xi)$, then $\bar{\bm u}\cdot \bm g(\tensor \Xi)=\overline{\bm u\cdot \bm
    g(\tensor\Xi)} =\bar{c}_m$ is also an element of $\bm c$. 

    
    In the numerical solution, it suffices to consider $c_m \nu e^{\lambda t} + \bar c_m \bar{\nu} e^{\bar{\lambda} t}$, where $\nu \in \mc$ denotes the middle element of the eigenvector $\bm v$. 
    For convenience, we denote $\nu = A + Bi, \lambda = C + Di, c_m = E + Fi$. Here, $A,B,C,D,E,F \in \mr$. Then,
    \begin{align*}
        c_m \nu e^{\lambda t} + \bar c_m \bar{\nu} e^{\bar{\lambda} t} &=  (E + Fi)(A + Bi) e^{ ( C + Di) t} +  (E - Fi)(A - Bi) e^{ ( C - Di) t} \\
        &= (P + Qi) e^{ ( C + Di) t} + (P - Qi) e^{ ( C - Di) t} \\
        &= \big(Pe^{Ct} + Qe^{Ct} i \big) e^{ Dt i} + \big(Pe^{Ct} - Qe^{Ct} i \big) e^{ - Dt i} \\
        &= \big(Pe^{Ct} + Qe^{Ct} i \big) [\cos(Dt) + \sin(Dt)i] \\
        &\quad + \big(Pe^{Ct} - Qe^{Ct} i \big) [\cos(Dt) - \sin(Dt)i] \\
        &= 2 Pe^{Ct} \cos(Dt) + 2 \big(Qe^{Ct} i) \sin(Dt) i \\
        &= 2 Pe^{Ct} \cos(Dt) - 2 Q e^{Ct} \sin(Dt) 
        \in \mr,
    \end{align*}
    among which $P = AE - BF$ and $Q = AF + BE$. 

\end{proof}

\begin{remark}
In practice, the imaginary part may be non-zero due to the round off error. 
In all numerical examples shown in this work, the magnitude of the imaginary part is extremely small (if it is non-zero), and we only keep the real part of the solution.
\end{remark}


\subsection{Algorithm summary}
\label{subsection:algo}

As a summary, ~\Cref{subsec:finite_approximation} to~\Cref{subsec:solution} present a numerical scheme that solves an autonomous ODE system using the eigendecomposition and a linear system solver. 
\Cref{subsec:adapt} introduces a heuristic adaptivity criterion to repeat the aforementioned procedure at appropriate time points to further enhance the accuracy.
We conclude the algorithm in~\Cref{algo:pseudo_code}.
\begin{algorithm}[htpb]
  \caption{Adaptive spectral Koopman method}
    \label{algo:pseudo_code}
    \begin{algorithmic}[1]
        \REQUIRE $n, T, N, \bm x_0, r, \gamma$
        \STATE{Set check points at $0=\tau_0 < \tau_1 < ... < \tau_n < T$.}
        \STATE{Let $L_i=x_0^i-r_i, U_i = x_0^i+r_i$ and set neighborhood $R_i$ as $R_i=[L_i+\gamma r_i, U_i-\gamma r_i]$ for $i = 1,2,...,d$, where $r_i=r$.} 
        \STATE{Generate Gauss-Lobatto points and differentiation matrix $\tensor D_i$ in $[L_i, U_i]$ for $i = 1,2,...,d$. 
        Construct collocation points $\bm\Xi$ for $d>1$ using the tensor product rule. (For $d=1$, $\bm\Xi$ is the set of the Gauss-Lobatto points.)}
        \STATE{Construct matrix $\tensor K$ using the formulas in~\Cref{subsec:finite_approximation}}
        \STATE{Compute eigen-decomposition $\tensor K\tensor V =\tensor V \tensor \Lambda$}
        \STATE{Solve linear system $\tensor V\tensor C=\bm g(\bm\Xi)$, where the $l$th column of matrix $\bm g(\bm\Xi)$ consists of the $l$th component of all collocation points (see~\Cref{subsec:solution}).}
        \FOR{$k = 1,2,3,\dotsc,n$}
        \STATE{Let $\nu_j$ be the middle element of the $j$th column of $\tensor V$. 
        Construct solution at time $\tau_k$ as $\bm g(\bm x(\tau_k)) = \displaystyle\sum_{j} \tensor C(j,:)\nu_j e^{\tilde\lambda_j (\tau_k-\tau_{k-1})}$, where $\tensor C(j,:)$ is the $j$th row of $\tensor C$.}
            \IF{$(\bm x(\tau_k))_i \notin R_i$ for any $i$}
                \STATE{Reset $L_i=x_i(\tau_k)-r_i, U_i=x_i(\tau_k)+r_i$ and $R_i=[L_i+\gamma r_i, U_i-\gamma r_i]$.}

                \STATE{Repeat steps $3-6$}
            \ENDIF
        \ENDFOR
        \RETURN {$\bm g(\bm x(T)) = \displaystyle\sum_{j} \tensor C(j,:)\nu_j e^{\tilde\lambda_j (T-\tau_{n})}$.}
    \end{algorithmic}
\end{algorithm}

In the ASK scheme, the neighborhood for all components must be updated in the adaptivity step.
This is because we set the current state of each component to be the midpoint of the corresponding neighborhood to avoid computing interpolation. 
Also, following the standard practice in the spectral method, we generate Gauss-Lobatto points $\xi_i$ and the associated differentiation matrix $\tensor D_i$ on $[-1, 1]$ first, and then scale them to $[L_i, U_i]$ as $\frac{U_i-L_i}{2} (\xi_i + 1 \big) + L_i$ and $\frac{2 \tensor D_i}{U_i - L_i}$ to improve the computational efficiency.
Moreover, it is worth emphasizing again that the isotropic setup (i.e., using the same number of Gauss-Lobatto points in each direction and fix $r_i\equiv r$) is not necessarily the optimal choice, and that the adaptivity in different directions may improve the efficiency of the algorithm. 
This is beyond the scope of this work and will be included in the future study.
\begin{remark}
The spectral method has been implemented to solving ODEs. 
The existing methods expand solution $\bm x(t)$ with orthogonal polynomials of $t$, which is again a discretization in time.
In this setting, when $\bff$ is nonlinear, one needs to solve a nonlinear system. 
Take a one-dimensional problem for example, the pseudo-spectral approach requires solving $\tensor D\bm y= \bff(\bm y)$, where $\tensor D$ is the differentiation matrix, $\bm y$ consists of the value of $\bm x(t)$ at collocation points (i.e., at different $t$), and $\bff(\bm y)$ is a vector of evaluating $\bff$ at $\bm y$. 
Therefore, the accuracy and efficiency rely on the property of $\bff$ as well as the performance of the nonlinear system solver, selection of initial points, etc.
In other words, even if a high order polynomial is used to approximate a smooth solution, the accuracy may be limited by the performance of the nonlinear solver.  
On the other hand, ASK uses discretization in space, and the accuracy and efficiency are influenced by the eigen-solver and the linear solver. These solvers are more mature and stable than nonlinear solvers in general and typically have (much) better guarantee in accuracy and efficiency.
\end{remark}


\section{Numerical Results}
\label{sec:experiments}

In this section, we first present the performance of ASK on six nonlinear ODE systems including $d=1,2,3$ in \Cref{subsec:solve_ode}.
In each example, we investigate the influence of number of Gauss-Lobatto points $N$, number of check points $n$, and the radius $r$ on the accuracy.
The reference solution is generated by Verner's ninth order Runge-Kutta (RK9) method~\cite{verner1978explicit} with sufficiently small time step if a close-form solution is not available.
Next, in \Cref{subsec:complexity_test} we compare the efficiency of ASK with conventional ODE solvers including Euler forward method, fourth order Runge-Kutta (RK4) method, five step Adams-Bashforth (AB5) method, and four step Adams-Moulton (AM4) method, since these are common methods used to solve ODEs. These comparisons include error against number of function calls, where the function refers to $\bff$ in the ODE. Also, we compare the error against running time (i.e., wall time) for different methods when evaluating $\bff$ is costly. 
Finally, \Cref{subsec:reuse_eigen} shows preliminary study on reusing computed eigenpairs and Koopman modes for solving new initial value problems. 
Here, we consider an uncertainty quantification problem with random initial condition for demonstration purpose. The mean and standard deviation of the solution are computed by ASK and RK4 to compare the performance.
Throughout the numerical examples, ASK employs the Chebyshev-Gauss-Lobatto points. Additionally, we also tested Legendre-Gauss-Lobatto points but there was no significant difference.
(All the MATLAB codes can be downloaded at \url{https://github.com/Navarro33/Adaptive-Spectral-Koopman-Method})

\subsection{Solving ODEs with ASK}
\label{subsec:solve_ode}
\subsubsection{Cosine model}
\label{subsec:cosine_model}
The cosine model is a synthetic model invented for our demonstrative purposes. The governing ODE is written as
  \[\frac{\dif x}{\dif t} = - 0.5 \cos^2(x).\]
We set $x(0)=\frac{\pi}{4}$ and $T=20$ in this example.
Despite the nonlinearity, the system has a closed-form solution $x(t) = \arctan(-0.5t + \tan(x_0))$. 
We aim to compute the solution at $T=20$.
The three experiments use the following parameters: 
\begin{enumerate}[\qquad (a)]
  \item test of $N$: $n = 200, r = \frac{\pi}{20}$; 
  \item test of $n$: $N=9, r=\frac{\pi}{20}$; 
  \item test of $r$: $n = 200, N = 9$. 
\end{enumerate}
In all these tests, we set $\gamma=0.2$. 
\Cref{fig:test_Cosine} summarizes these results in plots (a), (b) and (c), respectively.
The first test shows the exponential convergence of ASK with respect to $N$, which is similar to the conclusions in conventional spectral methods.
Test (b) shows that the accuracy does not change monotonically as $n$ varies given the parameter setting in this work. On the other hand, using no more than $100$ check points is sufficient to obtain good accuracy.
The last test illustrates that the accuracy shows a ``V shape'' with respect to
the radius, i.e., $r$ can not be too large or too small.
\begin{figure}[!ht]
  \centering
  \subfloat[Gauss-Lobatto points]{\label{fig:test_Cosine_N}
  \includegraphics[width=0.40\textwidth]{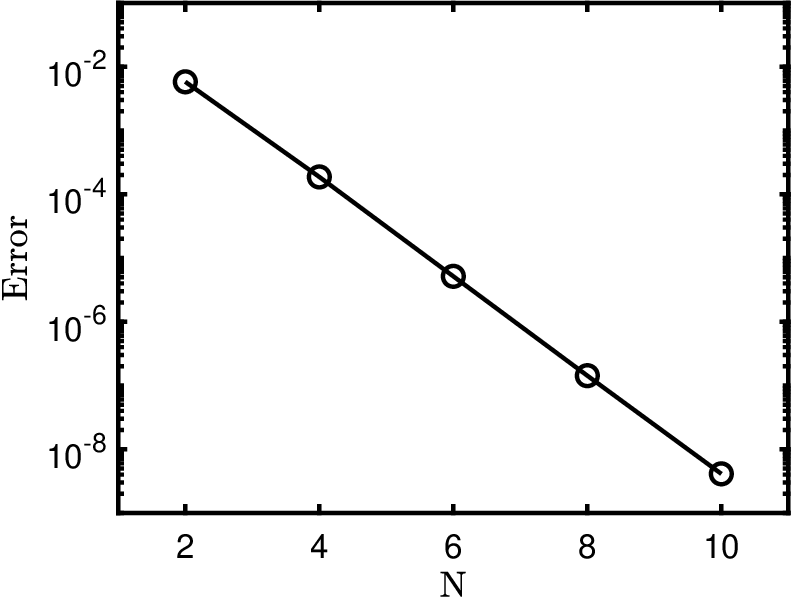}}\qquad
  \subfloat[Check points]{\label{fig:test_Cosine_n}
  \includegraphics[width=0.40\textwidth]{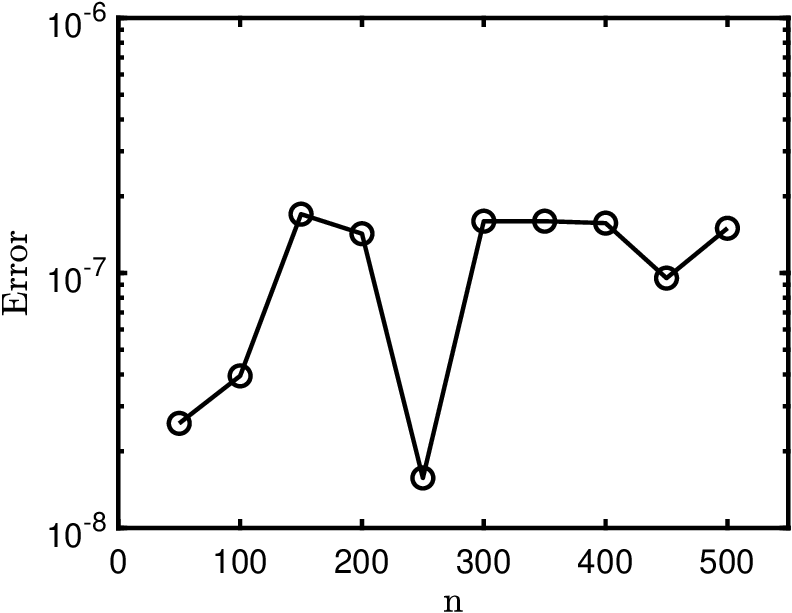}} \\
  \subfloat[Radius]{\label{fig:test_Cosine_r}
  \includegraphics[width=0.40\textwidth]{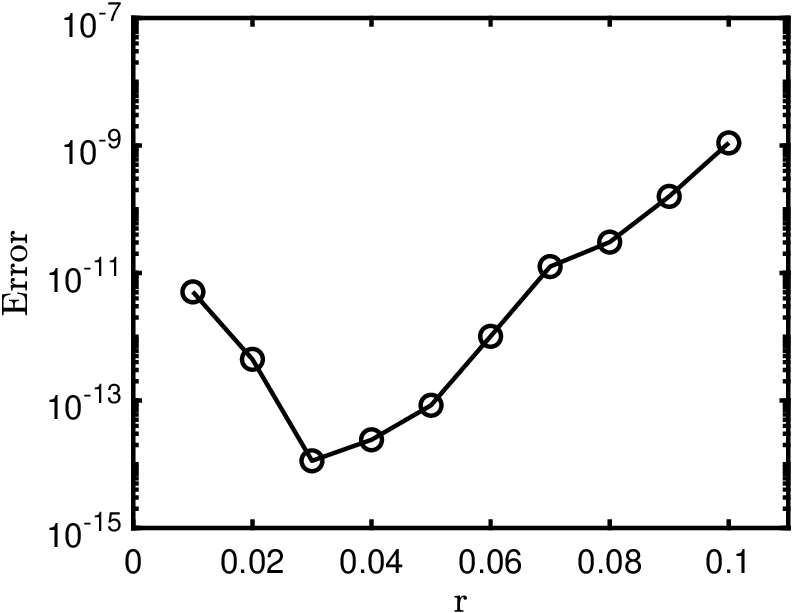}}
  \caption{Cosine model: (a) testing number of Gauss-Lobatto points $N$; (b) testing number of check points $n$; (c) testing radius $r$. } 
  \label{fig:test_Cosine}
\end{figure}


\subsubsection{Lotka-Volterra model}
\label{subsec:LV_model}
The Lotka-Volterra equations model the interactive evolution of the population of prey and predators \cite{bomze1983lotka}.  Specifically, it is defined by
\begin{align*}
    \frac{\dif x_1}{\dif t} = 1.1 x_1 - 0.4 x_1 x_2, \\
    \frac{\dif x_2}{\dif t} = 0.1 x_1 x_2 - 0.4 x_2.
\end{align*}
We set $\bm x(0)=(10, 5)^\top$ and $T=20$ in this example.
The parameters used in three different tests are as follows: 

\begin{enumerate}[\qquad (a)]
  \item test of $N$: $n = 200, r = 1.5$; 
  \item test of $n$: $N = 5, r = 1.5$; 
  \item test of $r$: $n = 200, N = 5$.
\end{enumerate}
In all the tests, $\gamma$ is set to 0.5.
Note that for multi-dimensional systems in the test of radius, all components share the same radius if it is not specified otherwise. \Cref{fig:test_LV} presents the results of these tests.
Similar to the cosine model, the error decreases exponentially with respect to $N$.
The accuracy is quite stable with respect to the number of check points in this case.
Furthermore, \Cref{fig:test_LV_r} shows that the radius cannot not be too small as in the first example. 

\begin{figure}[!ht]
  \centering
  \subfloat[Gauss-Lobatto points]{\label{fig:test_LV_N}
  \includegraphics[width=0.40\textwidth]{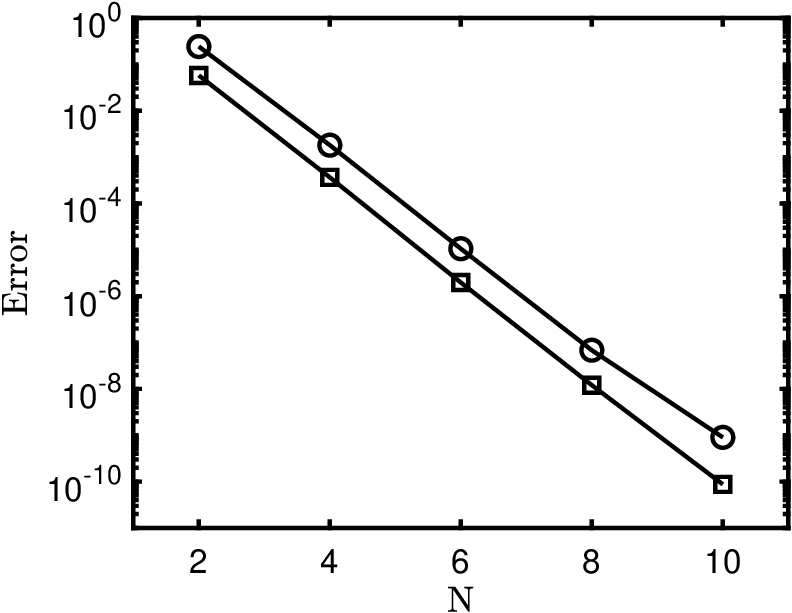}}\qquad
  \subfloat[Check points]{\label{fig:test_LV_n}
  \includegraphics[width=0.40\textwidth]{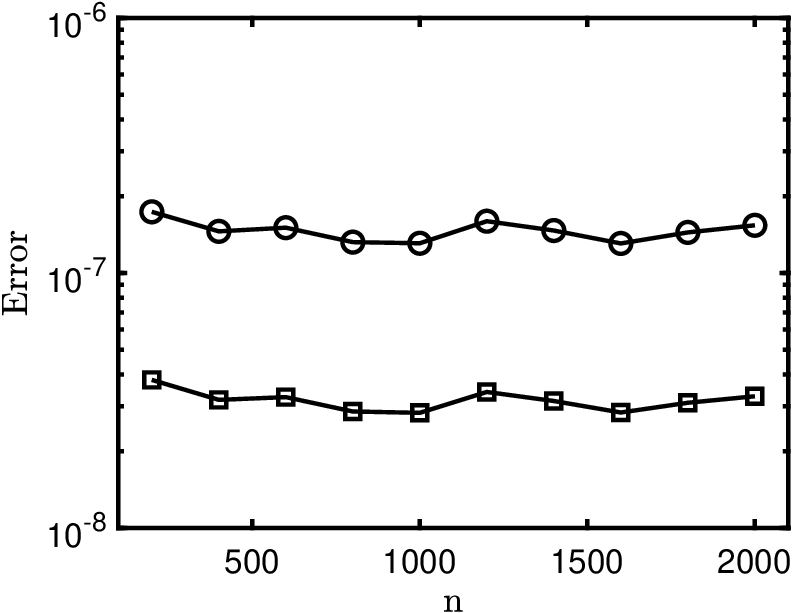}} \\
  \subfloat[Radius]{\label{fig:test_LV_r}
  \includegraphics[width=0.40\textwidth]{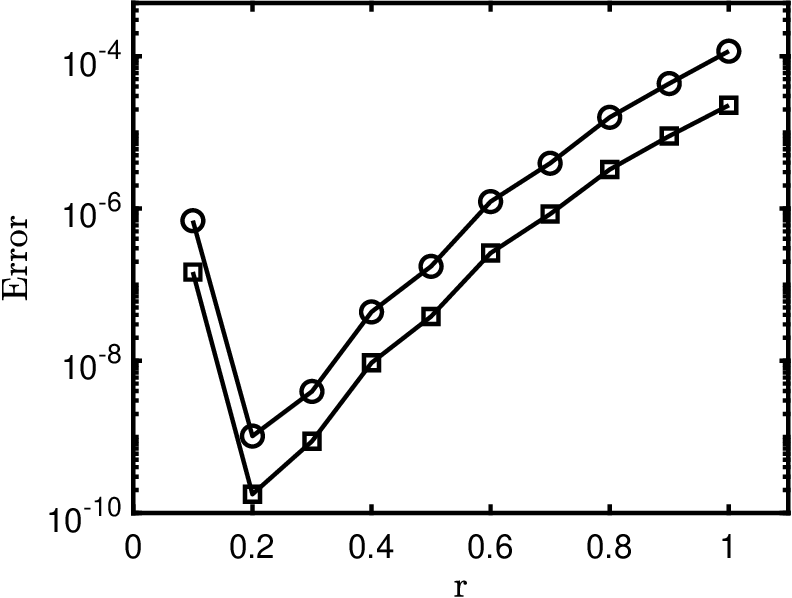}}
  \caption{Lotka-Volterra model: (a) testing number of Gauss-Lobatto points $N$
  (total number of collocation points is $(N+1)^2$); (b) testing number of check points $n$; (c) testing radius $r$. $\medcircle, \medsquare$ denote $x_1, x_2$, respectively.} 
  \label{fig:test_LV}
\end{figure}


\subsubsection{Simple pendulum}
\label{subsec:SP_model}
The simple pendulum is well studied in physics and mechanics. The movement of the pendulum is described by a second order ordinary differential equation,
\begin{align*}
    \frac{\dif^2 \theta}{\dif t^2} = - \frac{g}{L} \sin(\theta).
\end{align*}
Here, $\theta$ is the displacement angle, and $L$ denotes the length of the
pendulum. The parameter $g$ is the gravity acceleration. This second order
equation can be converted to a two-dimensional first-order ODE system. To keep the notations consistent, we define $x_1 := \theta$ and $x_2 := \frac{d \theta}{d t}$. Also, we set $L = g = 9.8$ in our numerical experiments. Correspondingly, 
\begin{align*}
    \frac{\dif x_1}{\dif t} &= x_2, \\
    \frac{\dif x_2}{\dif t} &= - \sin x_1.
\end{align*}
We set $\bm x(0)=(-\frac{\pi}{4},\frac{\pi}{6})^\top$ and $T=20$.
The parameters in the three tests are as follows: 
\begin{enumerate}[\qquad (a)]
 \item test of $N$: $n = 200, r = (\frac{\pi}{8}, \frac{\pi}{12})$;
 \item test of $n$: $N = 7, r = (\frac{\pi}{8}, \frac{\pi}{12})$; 
 \item test of $r$: $n = 200, N = 7$. 
\end{enumerate}
We set $\gamma = 0.2$ in all these tests. The results are presented in~\Cref{fig:test_SP}.
Again, we observe exponential convergence with respect to $N$ in~\Cref{fig:test_SP_N}. 
\Cref{fig:test_SP_n} implies that more check points can improve the accuracy but the difference is not very large.
\Cref{fig:test_SP_r} indicates there exists an ``optimal'' $r$ as in the
simple pendulum example. 
\begin{figure}[!ht]
  \centering
  \subfloat[Gauss-Lobatto points]{\label{fig:test_SP_N}
  \includegraphics[width=0.40\textwidth]{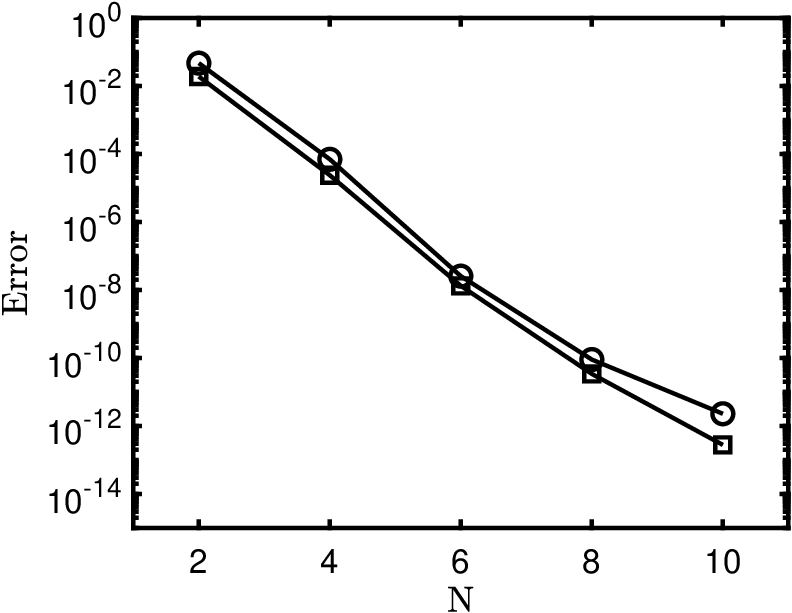}}\qquad
  \subfloat[Check points]{\label{fig:test_SP_n}
  \includegraphics[width=0.40\textwidth]{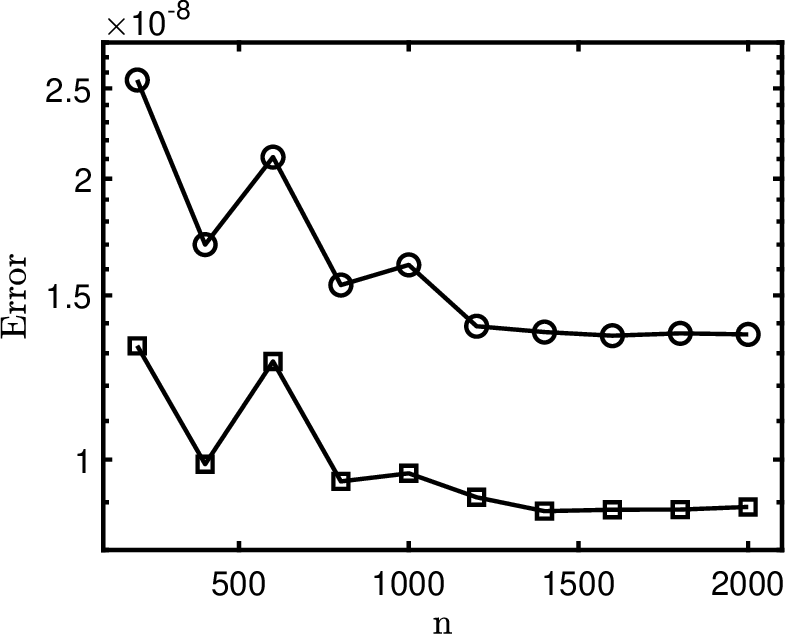}} \\
  \subfloat[Radius]{\label{fig:test_SP_r}
  \includegraphics[width=0.40\textwidth]{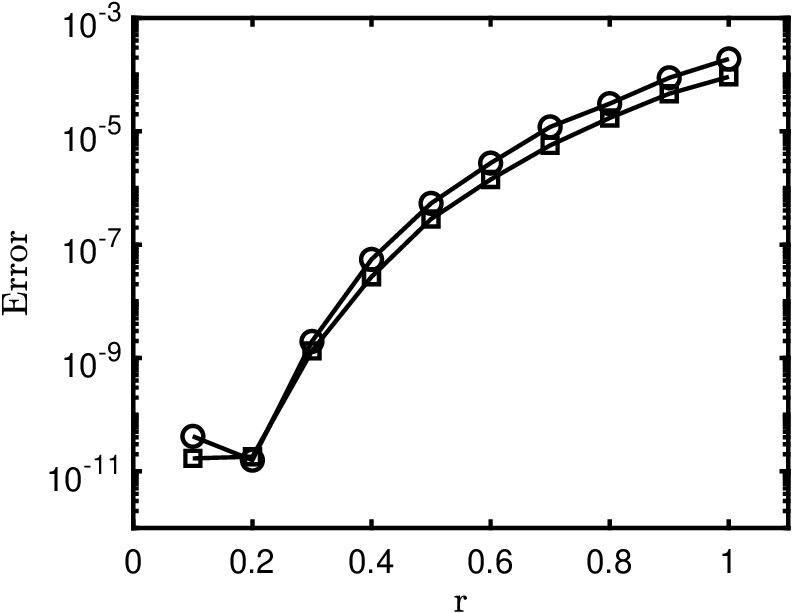}}
  \caption{Simple pendulum: (a) testing number of Gauss-Lobatto points $N$
  (i.e., $(N+1)^2$ collocation points in total); (b) testing number of check points $n$; (c) testing radius $r$. $\medcircle, \medsquare$ denote $x_1, x_2$, respectively.} 
  \label{fig:test_SP}
\end{figure}


\subsubsection{Limit cycle}
\label{subsec:LC_model}
The limit cycle is applied to model oscillatory systems in multiple research fields \cite{vidyasagar2002nonlinear}. Here, we follow the definition,
\begin{align*}
  &	\frac{\dif x_1}{\dif t} = - x_1 - x_2 +  \frac{x_1}{\sqrt{x_1^2 + x_2^2}}, \\
  &	\frac{\dif x_2}{\dif t} =   x_1 - x_2 +  \frac{x_2}{\sqrt{x_1^2 + x_2^2}}. 
\end{align*}
The closed-form solution is 
\begin{align*}
    x_{1}(t) &= \left[1 - \left(1-\sqrt{x_{1}(0)^2 + x_{2}(0)^2}\right) e^{-t} \right]\cos(t + \arctan(x_{2}(0)/x_{1}(0) )), \\
    x_{2}(t) &= \left[1 - \left(1-\sqrt{x_{1}(0)^2 + x_{2}(0)^2}\right) e^{-t} \right] \sin(t + \arctan(x_{2}(0)/x_{1}(0) )).
\end{align*}
We set $\bm x(0)=(\frac{\sqrt{2}}{2}, -\frac{\sqrt{2}}{2})^\top$ and $T=20$ in this example.
The parameters in the experiments are specified as follows: 
\begin{enumerate}[\qquad (a)]
 \item test of $N$: $n = 200, r = \frac{\sqrt{2}}{6}$; 
 \item test of $n$: $N = 7, r = \frac{\sqrt{2}}{6}$;
 \item test of $r$: $n = 200, N = 7$. 
\end{enumerate}
    We set $\gamma=0.2$ in all these tests.
The results shown in~\Cref{fig:test_LC} reveal similar patterns to the results of the simple pendulum, except that 
a very small $r$ can still lead to accurate results.

\begin{figure}[!ht]
  \centering
  \subfloat[Gauss-Lobatto points]{\label{fig:test_LC_N}
  \includegraphics[width=0.40\textwidth]{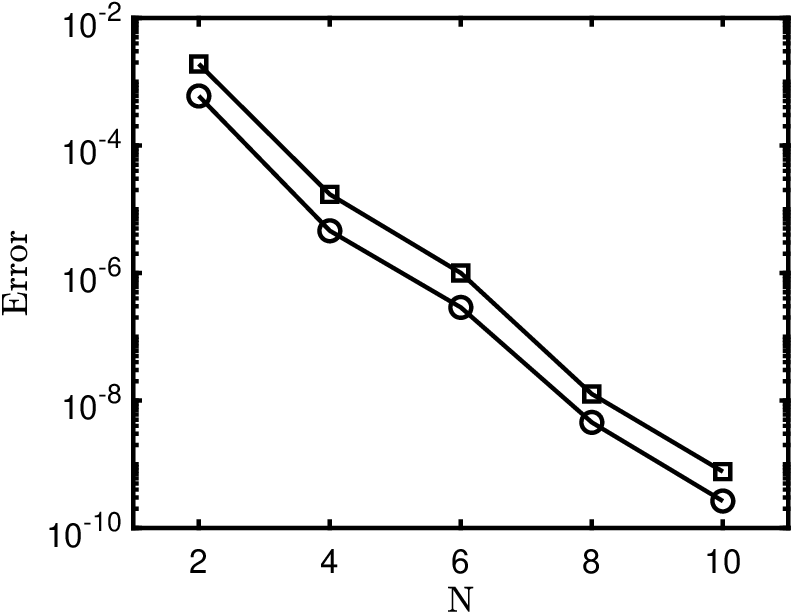}}\qquad
  \subfloat[Check points]{\label{fig:test_LC_n}
  \includegraphics[width=0.40\textwidth]{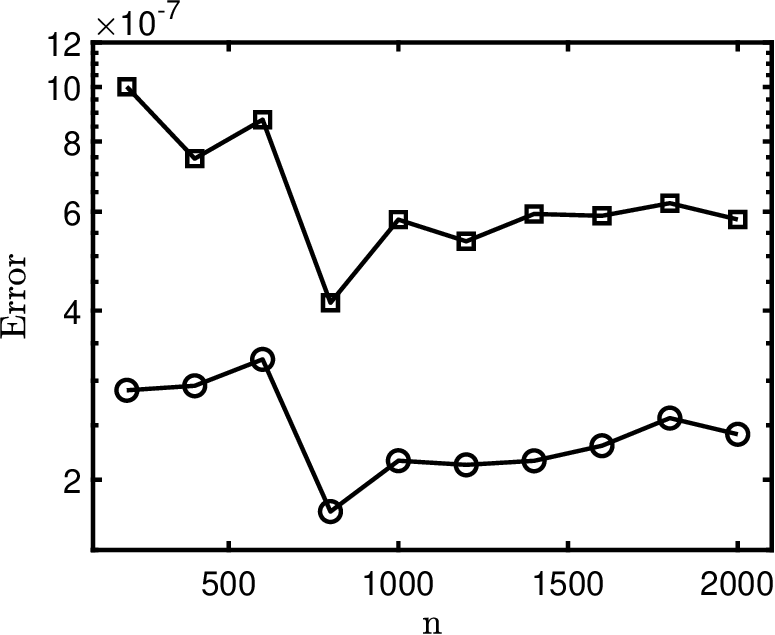}} \\
  \subfloat[Radius]{\label{fig:test_LC_r}
  \includegraphics[width=0.40\textwidth]{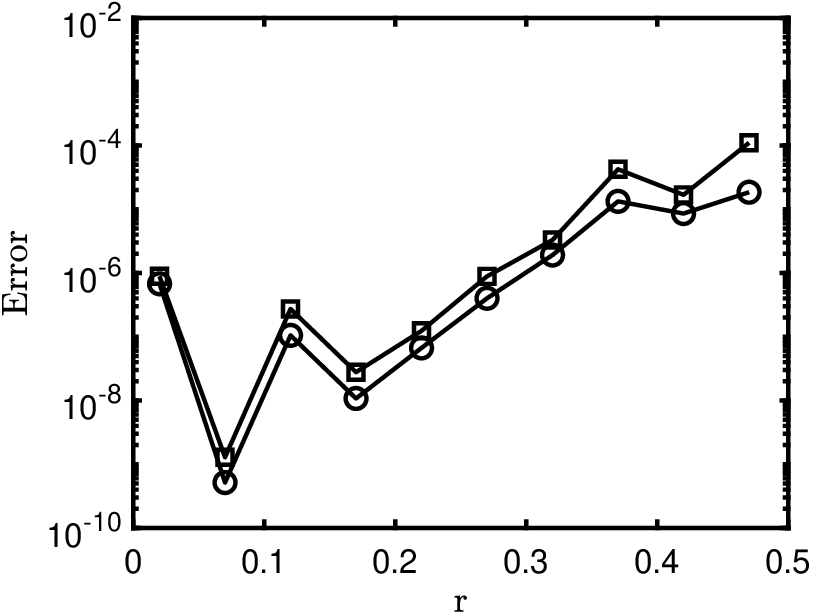}}
  \caption{Limit cycle: (a) testing number of Gauss-Lobatto points $N$ (i.e., $(N+1)^2$ collocation points in total); (b) testing number of checck points $n$; (c) testing radius $r$.  $\medcircle, \medsquare$ denote $x_1, x_2$, respectively.} 
  \label{fig:test_LC}
\end{figure}

For this example, we also compared ASK with RK4 at various time within $[0,T]$.
Given the closed-form solution $x_{C}(t)$, we computed the errors by $\vert x_{ASK}(t) - x_{C}(t) \vert$ and $\vert x_{RK4}(t) - x_{C}(t) \vert$. 
Here, RK4 employed $M = 200$ equidistant time points on $[0,T]$. 
The purpose of this comparison is to demonstrate that the meaning of the check points in ASK is different from the time grids in RK4 (and other conventional ODE solvers).
In this specific case, we set $n=M$. As for ASK, we used $N = 9, r = \frac{\sqrt{2}}{8}, \gamma = 0.2$ and the check points are set to be the same as the time points in RK4. 
With this set of parameters, ASK constantly outperforms RK4 significantly, as illustrated in \Cref{fig:test_full_LC_error}. 
For both components $x_1$ and $x_2$, the errors of ASK remain almost constant at the level of $10^{-10}$. 
In comparison, the error of RK4 exhibits a periodic pattern, rising slowly from $10^{-6}$ to $10^{-5}$. 
Moreover, \Cref{figt:test_full_LC_evolution} illustrates the evolution of the limit cycle model along time. The path decided by the two components elevates spirally as time evolves. If seen from above, the cross section is an exact circle.

\begin{figure}[!ht]
  \centering
  \subfloat[Error of $x_1$]{\label{fig:test_full_LC_error1}
  \includegraphics[width=0.40\textwidth]{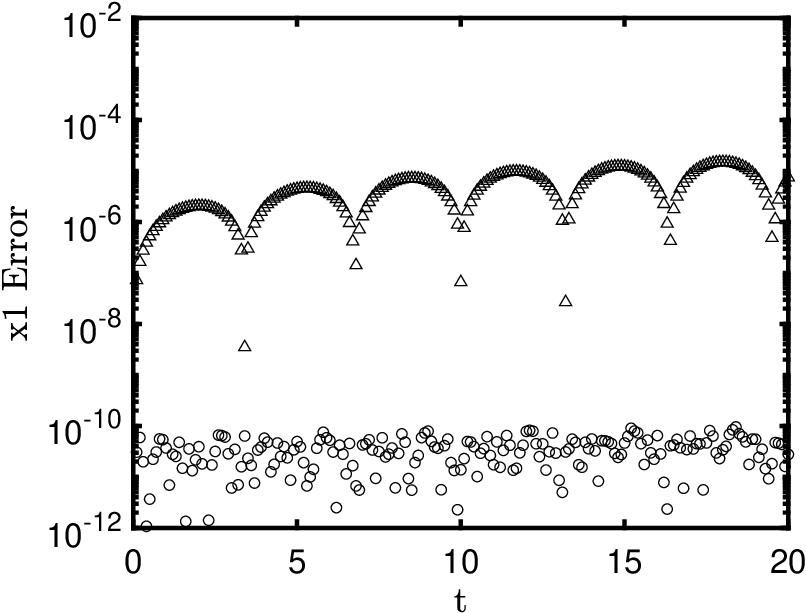}}\qquad
  \subfloat[Error of $x_2$]{\label{fig:test_full_LC_error2}
  \includegraphics[width=0.40\textwidth]{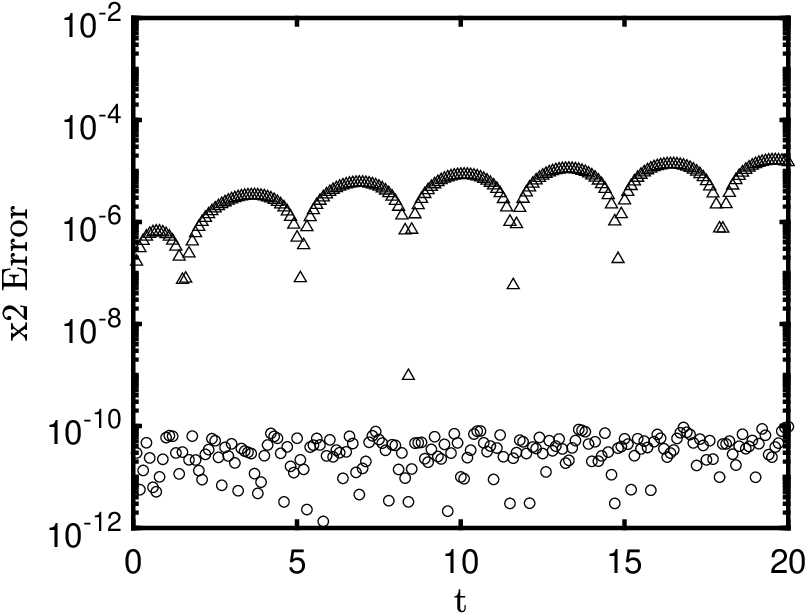}} 
  \caption{Error of the limit cycle solutions on $[0,T]$: $\medcircle$ denotes ASK and $\medtriangleup$ denotes RK4.}
  \label{fig:test_full_LC_error}
\end{figure}

\begin{figure}[!ht]
    \centering
    \includegraphics[width= 0.40\textwidth]{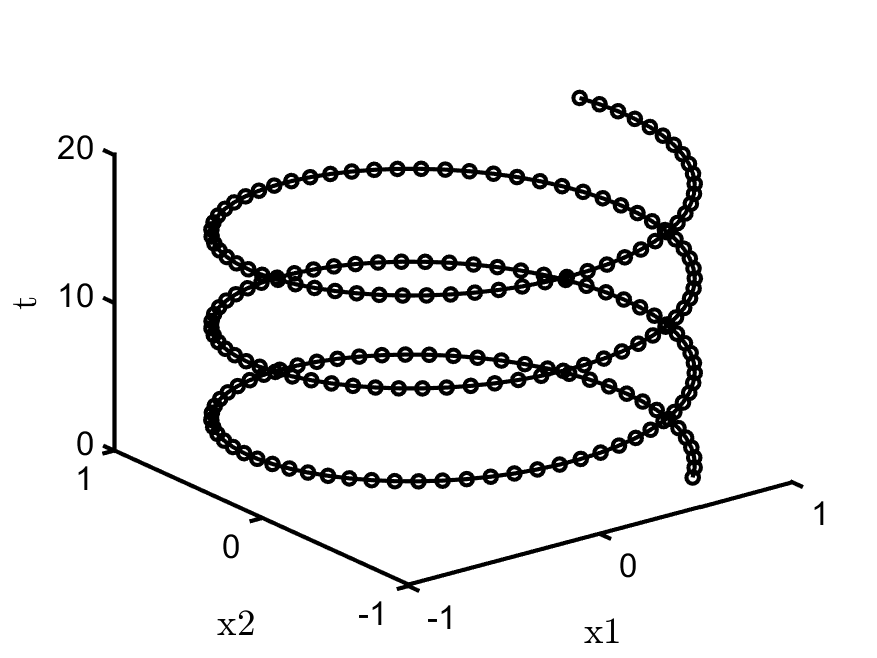}
    \caption{Limit cycle solution trajectory: $\medcircle$ denotes ASK and $-$ denotes the closed-form solutions.}
    \label{figt:test_full_LC_evolution}
\end{figure}


\subsubsection{Kraichnan-Orszag model}
\label{subsec:KO_model}
The Kraichnan-Orszag model comes from the problem raised in \cite{orszag1967dynamical}. This system is nonlinear and three-dimensional, defined by
\begin{align*}
  & \frac{\dif x_1}{\dif t} = x_2 x_3, \\
  & \frac{\dif x_2}{\dif t} = x_1 x_3, \\
  & \frac{\dif x_3}{\dif t} = -2 x_1 x_2.
\end{align*}
We set $\bm x(0) = (1, 2, -3)^\top$ and $T=20$.
In the three experiments, we employed the following parameters:
\begin{enumerate}[\qquad (a)]
\item test of $N$: $n = 400, r = 1$; 
\item test of $n$: $N = 3, r = 0.1$; 
\item test of $r$: $n = 400, N = 3$.
\end{enumerate}
Also, in all the tests, we set $\gamma=0.15$. The results are presented in~\Cref{fig:test_KO_N}. 
In particular, different from previous examples, \Cref{fig:test_KO_n} demonstrates that $n$ significantly influences the accuracy.  
This is because the Kraichnan-Orszag model exhibited strong oscillations, so it requires more frequent update of eigenpairs to guarantee high accuracy.
We can infer that there is an ``optimal'' $r$ in the Kraichnan-Orszag example, as demonstrated by \Cref{fig:test_KO_r}.

\begin{figure}[!ht]
  \centering
  \subfloat[Gauss-Lobatto points]{\label{fig:test_KO_N}
  \includegraphics[width=0.40\textwidth]{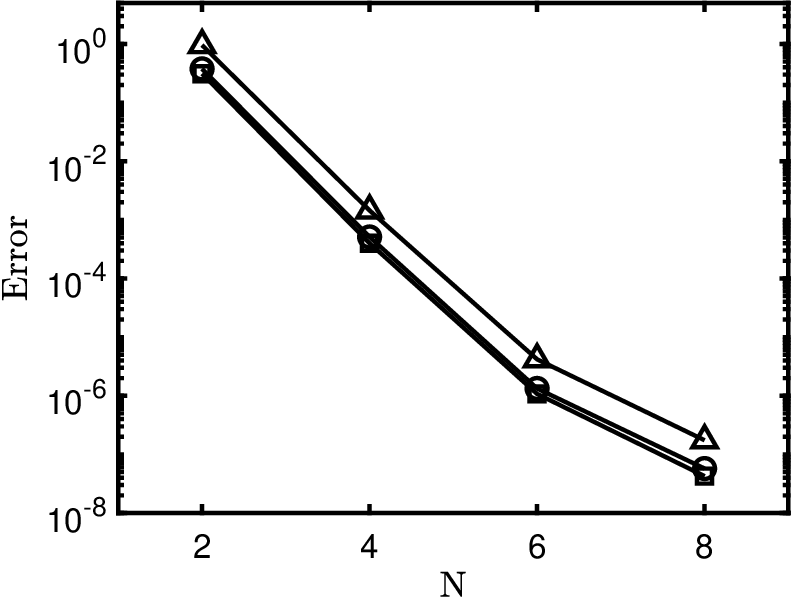}}\qquad
  \subfloat[Check points]{\label{fig:test_KO_n}
  \includegraphics[width=0.40\textwidth]{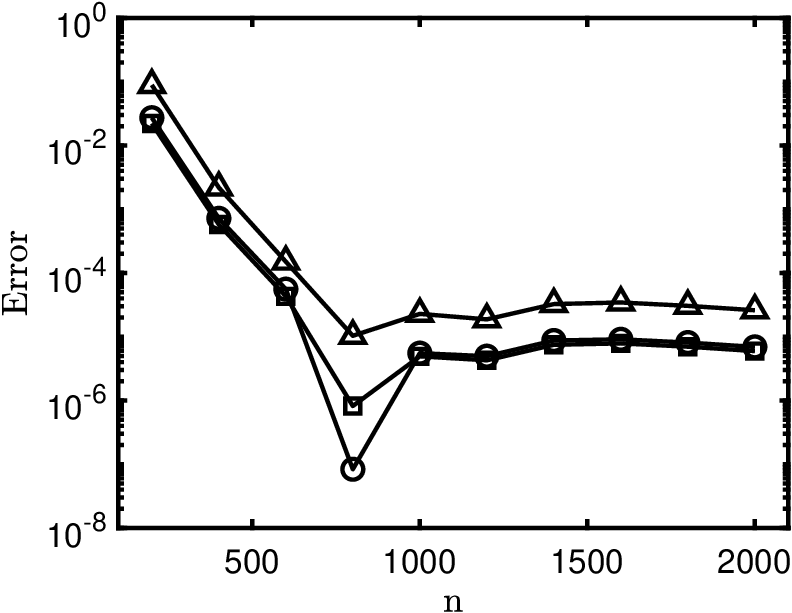}} \\
  \subfloat[Radius]{\label{fig:test_KO_r}
  \includegraphics[width=0.40\textwidth]{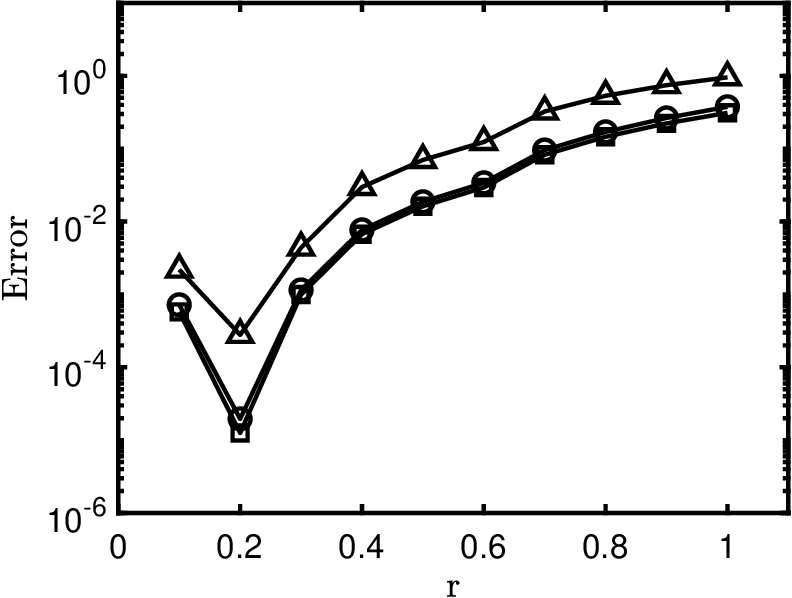}}
  \caption{Kraichnan-Orszag model: (a) testing number of Gauss-Lobatto points $N$ (i.e., $(N+1)^3$ collocation points in total); (b) testing number of check points $n$; (c) testing radius $r$. 
  $\medcircle, \medsquare, \medtriangleup$ denote $x_1, x_2, x_3$, respectively.} 
  \label{fig:test_KO}
\end{figure}


\subsubsection{Lorenz attractor}
\label{subsec:LO_model}
The Lorenz attractor was first introduced by Lorenz \cite{lorenz1963deterministic}. It is a highly chaotic system that models the turbulence in dynamic flows. The governing equations are as follows,
\begin{align*}
  &  \frac{\dif x_1}{\dif t} = 10 (x_2 - x_1), \\
  &  \frac{\dif x_2}{\dif t} = x_1 (28 - x_3) - x_2, \\
  &  \frac{\dif x_3}{\dif t} = x_1 x_2 - 3 x_3.
\end{align*}
We set $\bm x(0) = (5, 5, 5)^\top$ and $T=10$ in this example.
Parameters used in the experiments are listed here:
\begin{enumerate}[\qquad (a)]
 \item test of $N$: $n = 500, r = 4$; 
 \item test of $n$: $N = 5, r = 1$; 
 \item test of $r$: $n = 500, N = 5$.
\end{enumerate}
In all the tests, we set $\gamma=0.5$. The results are summarized in~\Cref{fig:test_LO}. 
As the Loren attractor exhibits chaotic behaviour, it requires a greater number of check points. 
Meanwhile, a relatively large radius favored the convergence of the algorithm.
This is probably because the eigenfunctions need to be approximated in a larger neighborhood of the solution to include sufficient information of the dynamics.

\begin{figure}[!ht]
  \centering
  \subfloat[Gauss-Lobatto points]{\label{fig:test_LO_N}\includegraphics[width=0.40\textwidth]{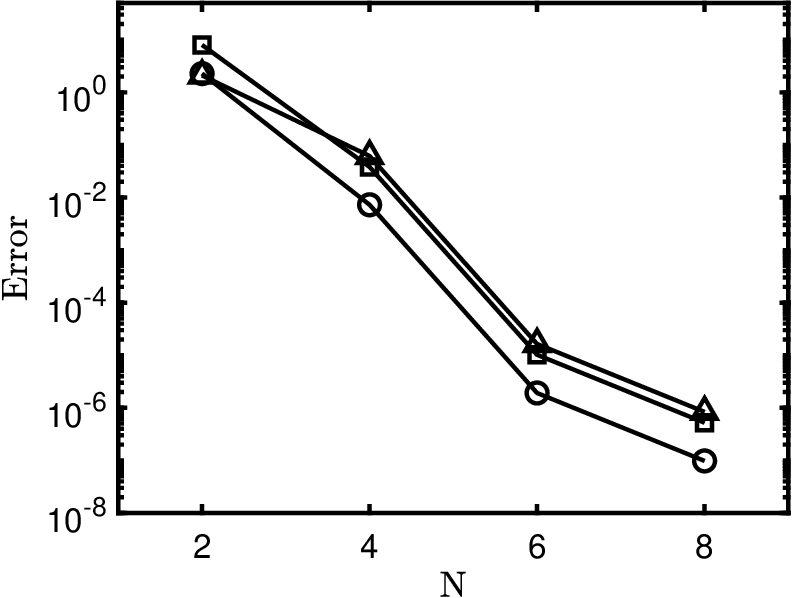}}\qquad
  \subfloat[Check points]{\label{fig:test_LO_n}\includegraphics[width=0.40\textwidth]{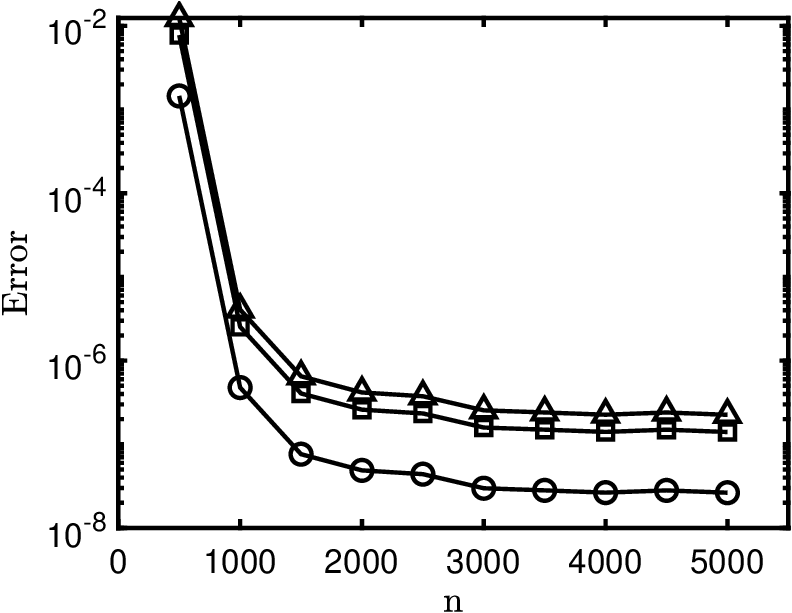}} \\
  \subfloat[Radius]{\label{fig:test_LO_r}\includegraphics[width=0.40\textwidth]{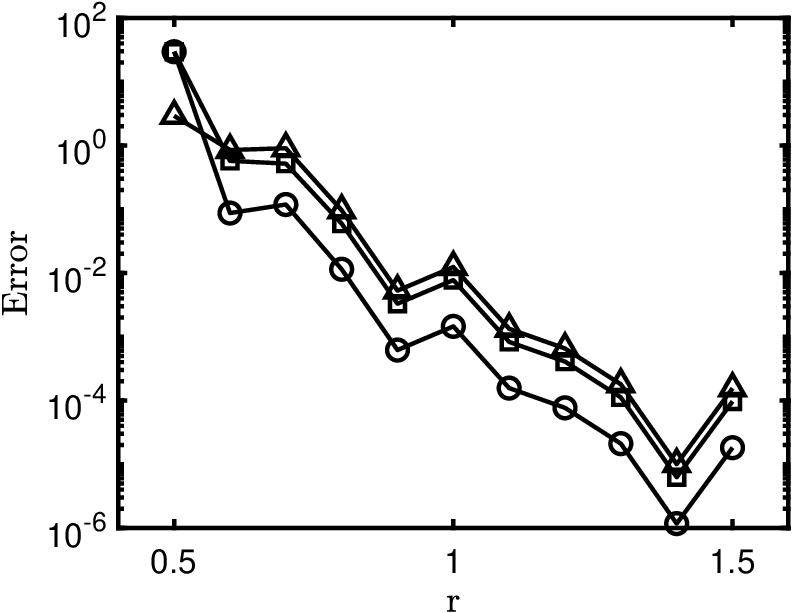}}
  \caption{Loren attractor: (a) testing number of Gauss-Lobatto points $N$ (i.e., $(N+1)^3$ collocation points in total); (b) testing number of check points $n$; (c) testing radius $r$. 
  $\medcircle, \medsquare, \medtriangleup$ denote $x_1, x_2, x_3$, respectively.} 
  \label{fig:test_LO}
\end{figure}

Next we compare the accuracy of ASK and RK4 to demonstrate the difference between the check points and time grids as in the Lorenz attractor example (see \Cref{subsec:LO_model}).
In this test, $T = 20$, and RK4 uses $M=2000$ time steps, i.e., step size $\Delta t=0.01$.
Since the Lorenz attractor does not have closed-form solutions, RK9 is used to compute the reference. 
To guarantee accuracy, RK9 used step size $\Delta t = 0.001$, i.e., $M=20000$ time steps. 
On the other hand, ASK was implemented with $n = 2000, N = 5, r = 1, \gamma = 0.75$.
For the comparison purpose, we set $n=M$ again and use small tolerance for the acceptable range. 
\Cref{fig:test_full_LO_error} reveals the accuracy of ASK in all three components. 
However, unlike the limit cycle case, the error increases as time evolves. 
Although it rises to around $10^{-3}$ at $t = 20$, ASK still yields an acceptable accuracy for such a chaotic system. 
In comparison, RK4's error ascends to a level that makes it impractical. 
To obtain an insight of how the Lorenz system evolves, we plot each of its component in \Cref{fig:test_full_LO_performance}. 
Up to time $t=10$, solutions given by ASK, RK4, and RK9 almost coincide. 
Nevertheless, RK4 deviates from the other two completely starting at $t=11$. 
The evolution vibrates violently and does not possess periodicity, which imposes difficulty on numerical solvers. 

\begin{figure}[!ht]
  \centering
  \subfloat[Error of $x_1$]{\label{fig:test_full_LO_error1}
  \includegraphics[width=0.40\textwidth]{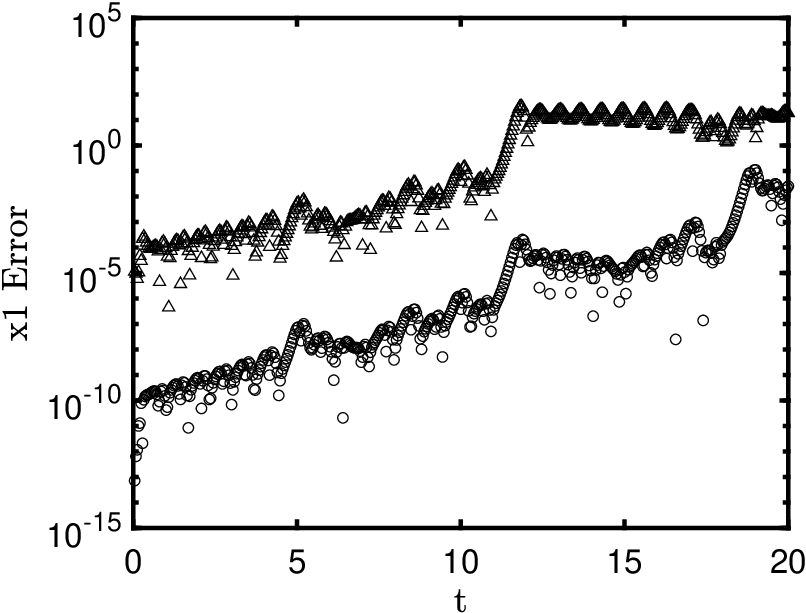}}\qquad
  \subfloat[Error of $x_2$]{\label{fig:test_full_LO_error2}
  \includegraphics[width=0.40\textwidth]{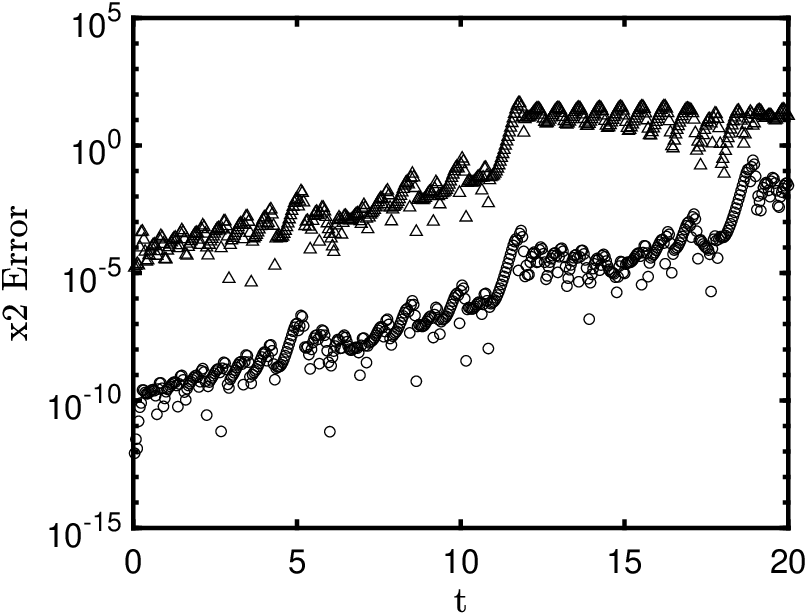}} \\
  \subfloat[Error of $x_3$]{\label{fig:test_full_LO_error3}
  \includegraphics[width=0.40\textwidth]{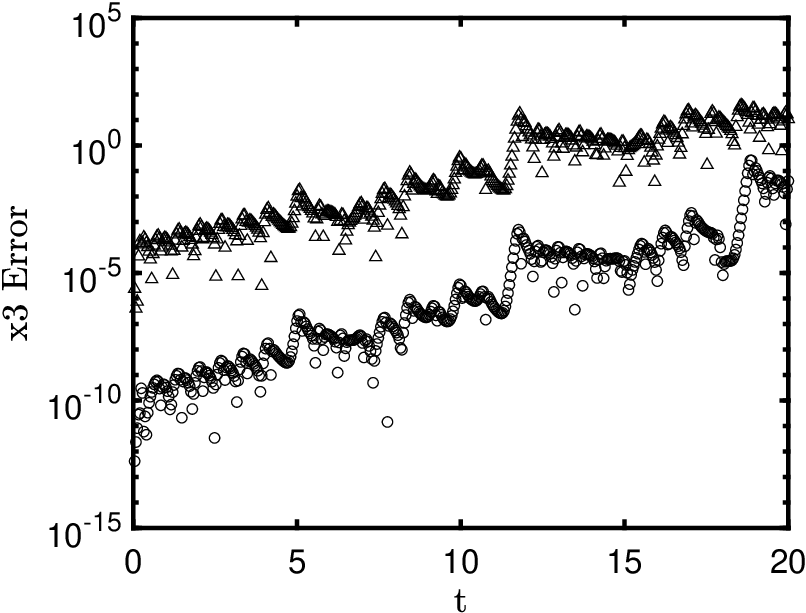}}
  \caption{Error of Lorenz attractor solutions on $[0,T]$: $\medcircle$ denotes ASK and $\medtriangleup$ denotes RK4.}
  \label{fig:test_full_LO_error}
\end{figure}

\begin{figure}[!ht]
  \centering
  \subfloat[$x_1$]{\label{fig:test_full_LO_x1}
  \includegraphics[width=0.40\textwidth]{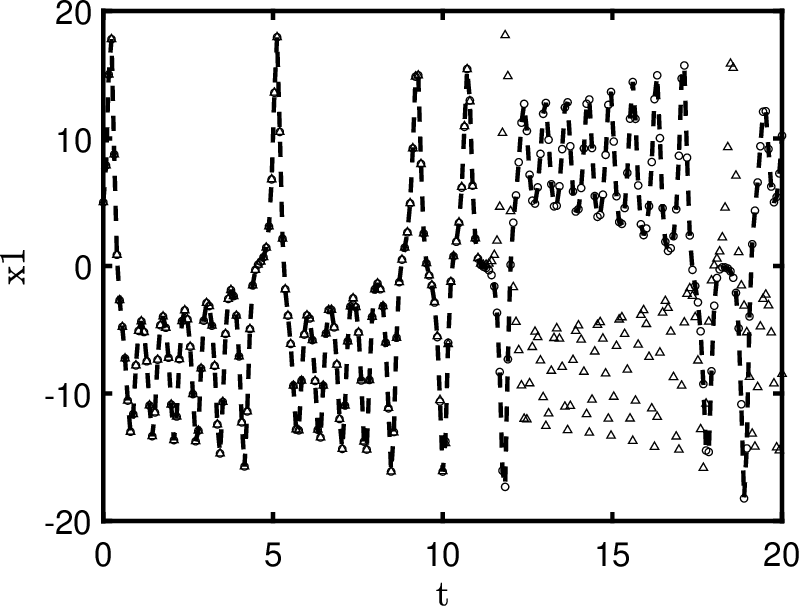}}~ 
  \subfloat[$x_2$]{\label{fig:test_full_LO_x2}
  \includegraphics[width=0.40\textwidth]{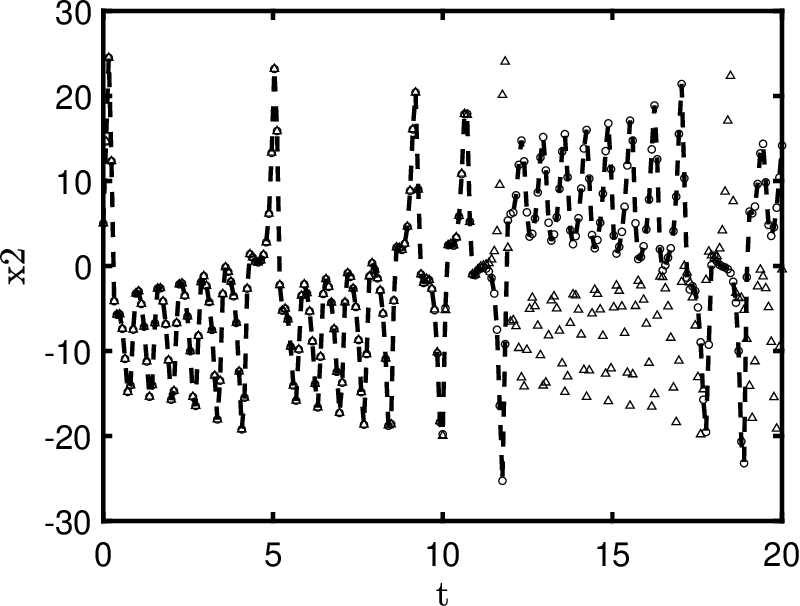}} \\
  \subfloat[$x_3$]{\label{fig:test_full_LO_x3}
  \includegraphics[width=0.40\textwidth]{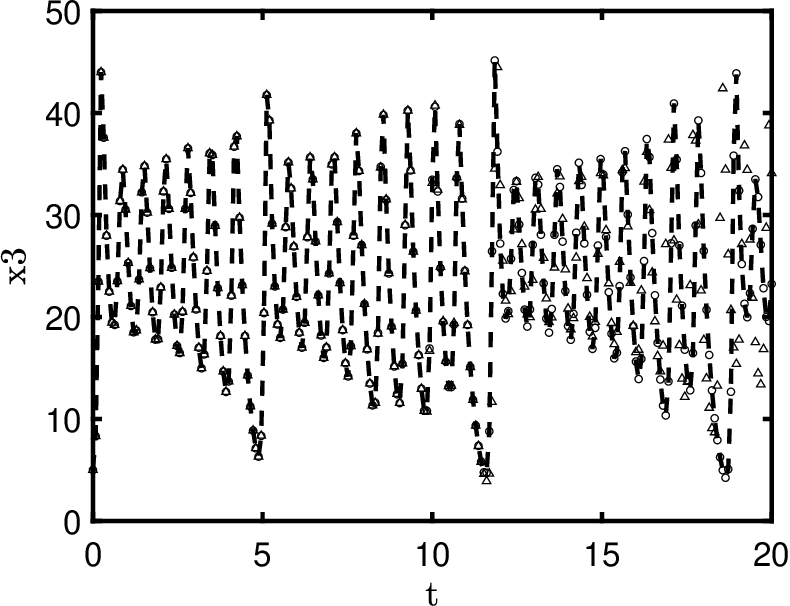}}
  \caption{Lorenz attractor evolution: $\medcircle$ and $\medtriangleup$ denote ASK and RK4, respectively; $--$ denotes the reference solutions given by RK9}
  \label{fig:test_full_LO_performance}
\end{figure}

The chaos can also be observed in a three-dimensional graph depicting the trajectory, using the numerical solutions given by ASK. As in \Cref{figt:test_full_LO_relation}, the lemniscate shape demonstrates the complexity of the system. 
\begin{figure}[!ht]
    \centering
    \includegraphics[width= 0.40\textwidth]{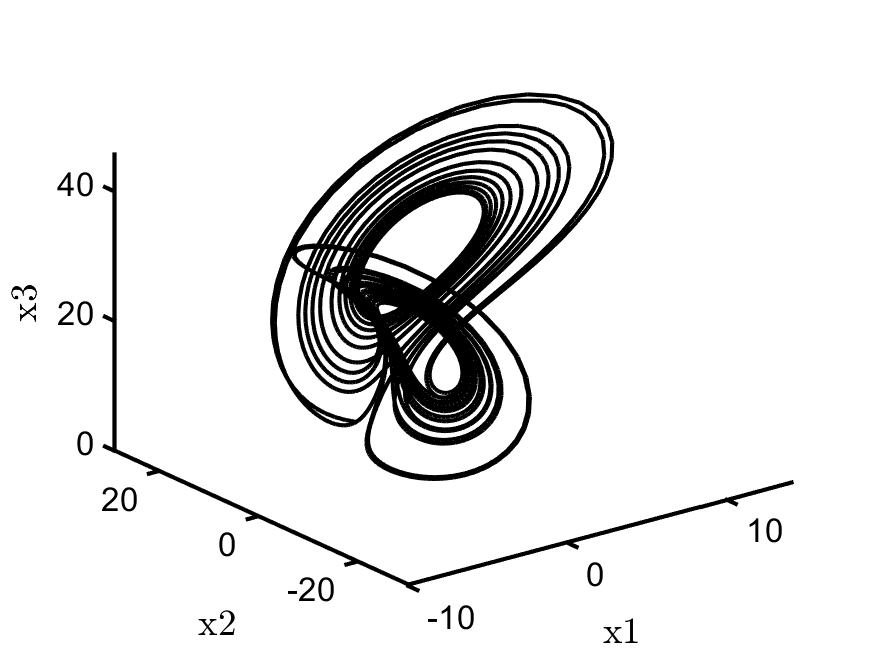}
    \caption{Lorenz attractor 3D visualization}
    \label{figt:test_full_LO_relation}
\end{figure}


\subsection{Computational complexity} 
\label{subsec:complexity_test}
By construction, the computational complexity of conventional explicit scheme solving ODE is $\mathcal{O}(M)$ where $M$ is the number of time steps. 
In other words, it is $M$ multiplied by a constant that represents the cost of evaluating function $\bff$ plus the cost of operations in each time step, which varies according to the accuracy of the scheme.
The computational complexity of ASK depends on the number of times that eigenfunctions are constructed (and corresponding eigenvalues as well as Koopman modes are computed). 
In this construction procedure, ASK needs to perform the eigendecomposition and solve a linear system.
For $d=1$ this is not costly because empirically we set $4\leq N\leq 10$, and the size of matrix in the eigendecomposition as well as the linear system is $N\times N$.
But when $d>1$, the complexity will increase exponentially with the dimension of the current setting because we use the tensor product rule to construct the collocation points and the matrix size is $(N+1)^d\times (N+1)^d$. 
Hence, ASK can be less efficient than conventional ODE solvers for high-dimensional systems. 

As an example, we present the accuracy and running time of different methods solving the simple pendulum problem (see~\Cref{subsec:SP_model}) in Table~\ref{tab:SP_model}.
Here, the final time $T = 20$, and we set $n = 200, N = 7, \bm x_0 = (-
\frac{\pi}{4}, \frac{\pi}{6}), r = (\frac{\pi}{8}, \frac{\pi}{12}), \gamma = 0.2$ in ASK.
For conventional ODE solvers, we set number of time steps as $m = 1000$. 
It is clear that explicit schemes RK4 and AB5 dramatically outperforms ASK in terms of computational time at the same accuracy level. 
Euler forward scheme is fast but not accuracy because it is a first-order scheme. AM4 is an implicit scheme that requires solving nonlinear systems in each step. 
Hence, it is $10$ times slower than ASK and is much slower than explicit schemes. 
But it has higher accuracy in this case.
\begin{table}[!ht]
  \centering
  \caption{Error and running time of solving the simple pendulum problem with $T=20$. Here, $n = 200, N = 7, \bm x_0 = (- \frac{\pi}{4}, \frac{\pi}{6}), r = (\frac{\pi}{8}, \frac{\pi}{12}), \gamma = 0.2$ in ASK and $m=1000$ (i.e., $\Delta t = 0.02$) for other ODE solvers.}
    \begin{tabular}{c|ccc}
        \hline\hline
        Algorithms & $x_1$    & $x_2$  & time (s) \\
        \midrule
        ASK   & 2.5524e-08 & 1.3242e-08 & 0.0498  \\
        Euler & 2.9415e-01 & 1.4698e-01 & 0.0008 \\
        RK4   & 9.3583e-09 & 1.3346e-08 & 0.0017  \\
        AB5   & 2.9335e-08 & 1.1873e-08 & 0.0019 \\
        AM4   & 1.6637e-09 & 6.4782e-10 & 0.5835 \\
        \hline\hline
    \end{tabular}%
    \label{tab:SP_model}
\end{table}%
Similarly, the comparison of different methods for Kraichnan-Orszag is presented in Table~\ref{tab:KO_model}, where $T = 20, n = 300, N = 5, \bm x_0 = (1, 2, -3), r = (0.2, 0.2, 0.2), \gamma = 0.15$ for ASK, and number of time steps $m = 3000$ for other ODE solvers.
In this test, ASK has the best accuracy but it is much slower than the explicit schemes. 
The gap between the computational time is larger than that in the simple pendulum problem.
Also, the computation time of ASK is only slightly shorter than that of AM4.
This is because the Kraichnan-Orszag problem is three-dimensional, and, as expected, ASK becomes less efficient.
\begin{table}[!ht]
  \centering
  \caption{Error and running time of the Kraichnan-Orszag model with $T=20$. Here, n = $300, N = 5, \bm x_0 = (1, 2, -3), r = (0.2, 0.2, 0.2), \gamma = 0.15$ for ASK and $m=3000$ (i.e., $\Delta t = 1/150$) for other ODE solvers. }
  \begin{tabular}{c|cccc}
    \hline\hline
    Algorithms & $x_1$    & $x_2$    & $x_3$ & time (s)\\
    \midrule
    ASK   & 3.0384e-08 & 2.3718e-08 & 8.4070e-08 & 2.1840 \\
    Euler & 8.8547e-01 & 3.2547e-01 & 4.7061e+00 & 0.0082  \\
    RK4   & 2.1203e-07 & 1.6047e-07 & 6.7413e-07 & 0.0113 \\
    AB5   & 8.3518e-06 & 6.6129e-06 & 2.5169e-05 & 0.0231 \\
    AM4   & 4.8154e-07 & 3.8129e-07 & 1.4022e-06 & 2.5995 \\
    \hline\hline
    \end{tabular}%
    \label{tab:KO_model}
\end{table}%

However, in the comparisons above, the cost of evaluating $\bff$ in
the dynamical system is extremely low. 
In the next comparison, we consider an evaluation of $\bff$ as a function call,
and compare the accuracy of ASK and the explicit solvers against number of function calls.
The cosine model ($d=1$), the simple pendulum ($d=2$), and the Kraichnan-Orszag problem ($d=3$) exemplify the comparison. The results are provided in~\Cref{fig:test_efficiency}. 
Here, the error in the simple pendulum case was computed by $\sqrt{\frac{e_1^2 + e_2^2}{2}}$, where $e_1, e_2$ are the errors in $x_1, x_2$, respectively. 
Similarly, the error for the Kraichnan-Orszag model is $\sqrt{\frac{e_1^2 + e_2^2 + e_3^2}{3}}$.
In this test, ASK starts with a small $N$ and keeps increasing it by $2$ as in the convergence study in~\Cref{subsec:solve_ode}.
For conventional ODE solvers, we start with a large time step and then keep reducing it by half.
Figure~\ref{fig:test_efficiency}(a) indicates that ASK is superior to all
conventional solvers even RK9 for the cosine model ($d=1$).
For the simple pendulum ($d=2$), RK9 is the most efficient method, while ASK outperforms RK4 and AB5 when number of function calls is beyond $2000$. 
For the Kraichnan-Orszag model ($d=3$), ASK is less efficient than high-order explicit schemes and can only outperform the Euler forward method. 
These phenomena are consistent with the discussion at the beginning of this subsection. Of note, we do not include conventional implicit solvers in this comparison as they are slower than the explicit solvers with the similar accuracy level for the examples we consider in this work.
\begin{figure}[!ht]
  \centering
  \subfloat[Cosine model]{
  \includegraphics[width=0.32\textwidth]{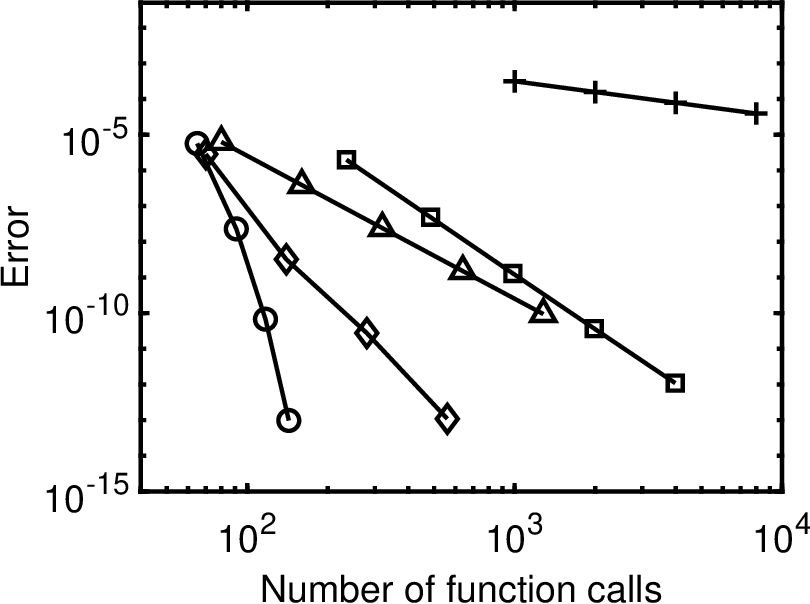}}~
\subfloat[Simple pendulum]{
  \includegraphics[width=0.31\textwidth]{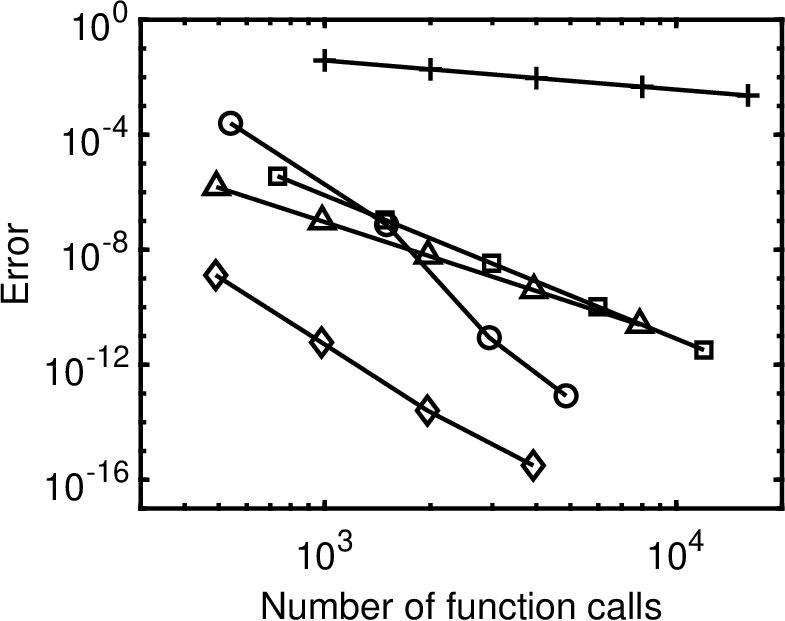}}~
\subfloat[Kraichnan-Orszag model]{
  \includegraphics[width=0.32\textwidth]{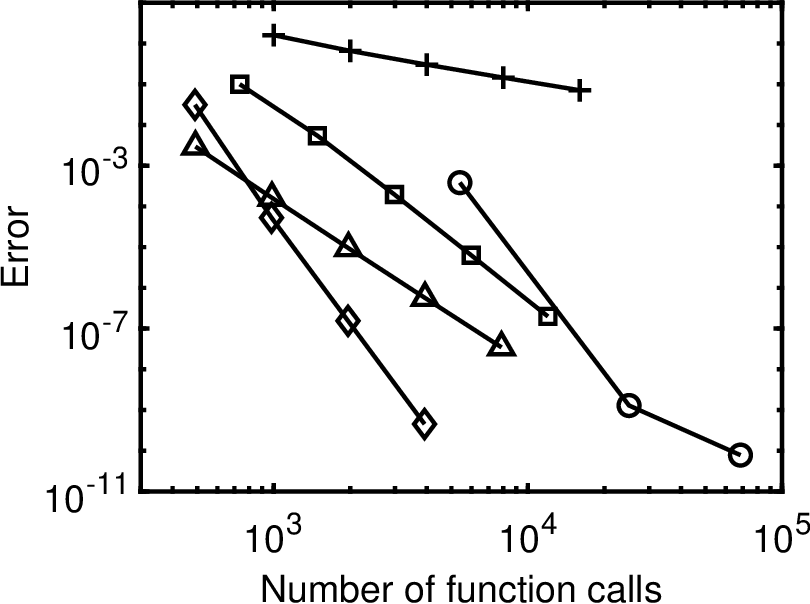}}
  \caption{Comparison of computational efficiency: error against number of function calls. Symbols $\medcircle, +, \medtriangleup, \meddiamond, \medsquare$ denote ASK, Euler, RK4, RK9, and AB5, respectively.}
  \label{fig:test_efficiency}
\end{figure}

We note that the comparison of error against number of function calls still can not fully reflect the efficiency of the algorithms.
It seems to be straightforward that the total computational time of evaluating $\bff$ is the time of evaluating $\bff$ once multiplied by the number of function calls.
However, this is not necessarily true in the modern computing tools.
For example, in MATLAB and Python, built-in vectorization or tensorization approaches are used to accelerate the computing. 
In other words, evaluating $\bff$ at different collocation points $\bm x$ can be vectorized and be achieved with one function call instead of using a for-loop to evaluate $\bff$ at each collocation points one by one. 
Even though the computational time for this vectorized function call is longer than evaluating $\bff$ at one collocation point, it can be much shorter than using a for-loop.
Consequently, ASK is more efficient than conventional ODE solvers when evaluating $\bff$ is costly.
To demonstrate this advantage, we artificially slow down the evaluation of $\bff$ in the above three tests and compare the error against computational time in different methods.
Specifically, for the cosine model and the simple pendulum, we replace sine and cosine functions with their corresponding Taylor expansions up to $x^{1000}$ (i.e., $500$ terms in the expansion); for the Kraichnan-Orszag model, we evaluate $\bff$ $1000$ times in the code before output its value. 
In this way, the computational time for evaluating $\bff$ increases significantly.
We repeat the same tests as in the error against number of function calls study. The results of error against computational time are presented in Figure~\ref{fig:test_efficiency_time}.
It is observed that ASK outperforms all explicit solvers (even RK9) in the selected error and time ranges.
The advantage of ASK over the conventional solvers becomes less significant as the dimension increases, which is consistent with the discussion on the complexity.
Comparing the results in Figure~\ref{fig:test_efficiency_time} and the results in Tables~\ref{tab:SP_model} and~\ref{tab:KO_model}, we can see the impact of cost of evaluating $\bff$ on the efficiency of different approaches as we are solving the same two- and three-dimensional problems.
These comparisons indicate that ASK can be more efficient than conventional ODE solvers when evaluating $\bff$ is costly because of the build-in parallel mechanism for evaluating $\bff$ at multiple $\bm x$.
\begin{figure}[!ht]
  \centering
  \subfloat[Cosine model]{
  \includegraphics[width=0.32\textwidth]{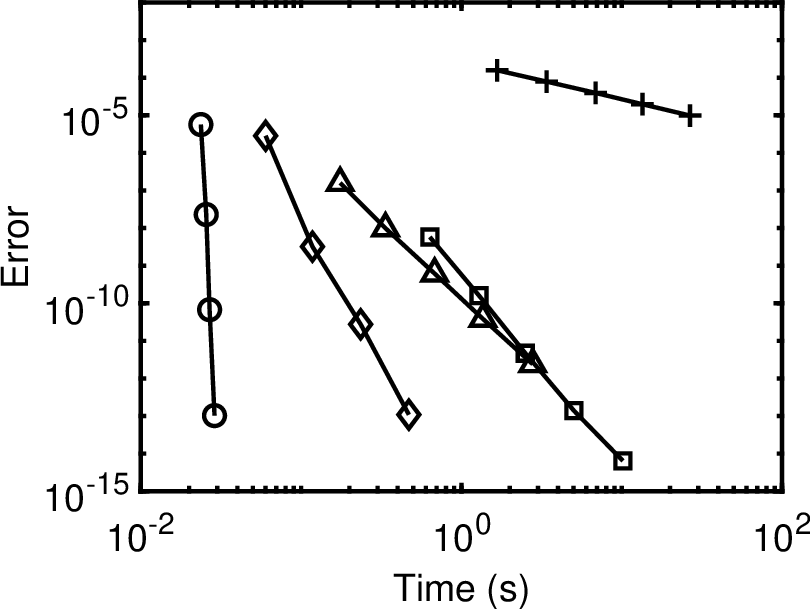}}~
\subfloat[Simple pendulum]{
  \includegraphics[width=0.31\textwidth]{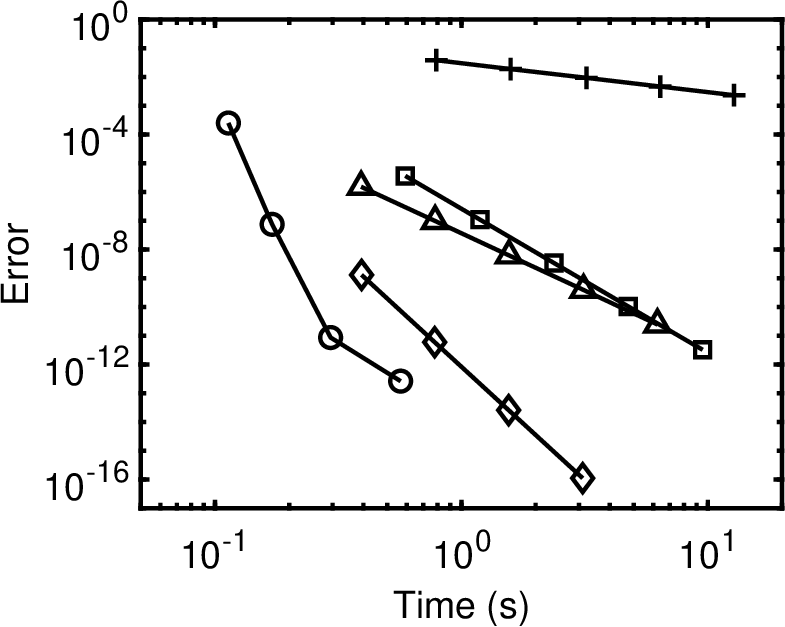}}~
\subfloat[Kraichnan-Orszag model]{
  \includegraphics[width=0.31\textwidth]{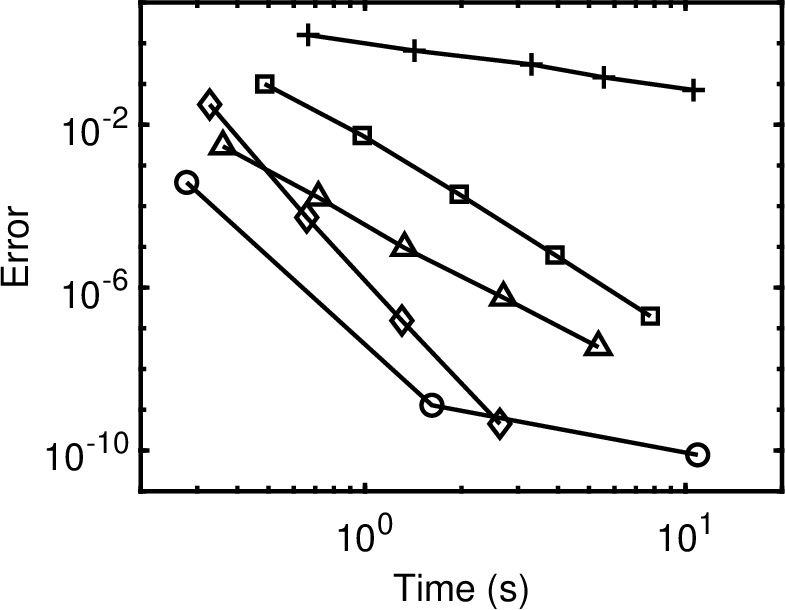}}
  \caption{Comparison of computational efficiency: error against running time.
  Symbols $\medcircle, +, \medtriangleup, \meddiamond, \medsquare$ denote ASK, Euler, RK4, RK9, and AB5, respectively. The time of evaluating $\bff$ is artificially increased in the code.}
  \label{fig:test_efficiency_time}
\end{figure}

\subsection{Reusing eigenpairs and Koopman modes}
\label{subsec:reuse_eigen}
It is typically necessary to solve an ODE with different initial values in the studying the property
of its dynamics numerically, such as sensitivity analysis, perturbation analysis, uncertainty quantification (UQ), etc.
In this case, another advantage of ASK in computation is that it can potentially reuse computed eigenpairs and Koopman modes to facilitate solving the same ODE with different initial values.
Specifically, if~\eqref{eq:trun_expansion} is obtained via ASK based on initial value $\bm x_0$.
Then, for another initial value $\bm x_1$ lying in a sufficiently small neighborhood of $\bm x_0$ (e.g., an open ball centered at $\bm x_0$), it is possible to directly write down the solution as $\sum_{j=0}^N \tilde c_j \vphin_j(\bm x_1) e^{\tilde \lambda_j t}$. 
Here, the only additional computation is evaluating $\vphin_j(\bm x_1)$ for each $j$, which is accomplished by Lagrange interpolation since we computed values of $\vphin$ at the collocation points via eigen-decomposition.
Specifically, the $j$th column of matrix $\bm V$ consists of the values of $\vphin_j$ at the collocation points (see~\Cref{subsec:eigen_decomposition}).
The applicability of this idea relies on the property of the dynamical system and more comprehensive study is needed to decided the radius of the neighborhood for desired accuracy at time $t$. 

Even though a systematic study is beyond the scope of this paper, we present an illustrative example to show the potential of applying ASK to solve an ODE with different initial values efficiently.
Here, ASK solves a dynamical system with random initial values for UQ study. 
Our goal is to compute the mean and the standard deviation of the solution at time $T$. 
The cosine model is used here for demonstration, where the cosine function is replaced with its Taylor expansion as in~\Cref{subsec:complexity_test}. 
Here, Monte Carlo (MC) simulation is leveraged to estimate the mean and standard deviation of the solution at $T=1$, as MC is a state-of-the-art sampling-based UQ method.  
The initial value is set as $\bm x_0 = \frac{\pi}{4}\theta$, where $\theta\sim U[0.75, 1.25]$ is a uniform random variable, and we generate 5,000 samples of $\theta$ denoted as $\theta^{(1)}, \theta^{(2)},\dotsc, \theta^{(5000)}$.
ASK first solves a deterministic ODE with $\bm x_0=\frac{\pi}{4}$ using parameters $N=8, r=0.2$, after which Lagrange interpolation is applied to evaluate $\vphin_j$ ($0\leq j\leq N$) at all the samples of the initial values, i.e., $\frac{\pi}{4}\theta^{(i)}$ ($0\leq i\leq 5000$), to directly construct the solutions.
Then, we use the empirical mean and standard deviation of these 5,000 solutions to estimate the mean and standard deviation of the ODE. Particularly, RK4 serves as a prototypical example of explicit solvers, and set $\Delta t = 0.1$ to solve $5,000$ initial value problems to obtain samples of the solution.
Subsequently, the empirical mean and standard deviation are computed for estimation.
We repeat these tests for $1,000$ sets of independent samples of $\theta$, and present the results in Table~\ref{tab:stat_error}.
It shows that the accuracy of ASK and RK4 is similar for this problem, but the time for solving 5,000 initial values problems (i.e., Average time in the table) indicate that ASK is much more efficient.
Of note, the RK4 implemented here is a vectorized version solving all initial value problems simultaneously, which is much faster than a for-loop of $5,000$ iterations.
The error of RK4 for solving each initial value problem is at the level of $10^{-6}$, which is sufficiently small for estimating statistics in this case because the statistical error is at the level of $10^{-3}$. 
Also, for demonstration purpose, this example does not activate the adaptivity step in ASK, so we only perform eigen-decomposition and solve the linear system once.
A more systematic study and delicate algorithm design will be included in our future work. 
\begin{table}[!ht]
  \centering
  \caption{Relative errors of estimating mean and standard deviation with random initial values using 5,000 Monte Carlo simulations. The relative error is computed by dividing the absolute error by the reference.}
    \begin{tabular}{c|cccc}
    \hline\hline
    Algorithms & Error of mean  & Error of std & Average time (s)\\
    \midrule
    ASK & 6.1403e-03 $\pm$ 4.5974e-03 & 7.4006e-03 $\pm$ 5.5312e-03 & 0.0031 \\
    RK4 & 6.1403e-03 $\pm$ 4.5977e-03 & 7.4014e-03 $\pm$ 5.5306e-03 &  1.0555 \\
    \hline\hline 
    \end{tabular}%
    \label{tab:stat_error}
\end{table}%


\section{Conclusion and Discussion}
\label{sec:dicussion_conclusion}
The ASK method uses the spectral-collocation (i.e., pseudo spectral) method in the state space instead of in time to solve nonlinear autonomous dynamical systems. 
It discretizes the generator of Koopman operator and employs the eigendecomposition to obtain approximation of the eigenfunctions and eigenvalues to construct solutions.
Therefore, like the spectral method, ASK is an expansion-based method to solve ODE systems, in which the basis functions in the expansion are approximated eigenfunctions of the Koopman operator.
In each numerical example presented in this work, ASK exhibits \emph{exponential convergence} as the conventional spectral method.
Therefore, it is suitable for the circumstances where high-accuracy solutions are desired and $\bff$ is expensive to evaluate. 
Different from existing ODE solvers that obtain solutions on mesh grids, ASK does not need a time mesh and can evaluate the solution at any time. 
Hence, the resolution of the time mesh which impacts the solutions of conventional ODE solvers like Runge-Kutta methods does not influence ASK. 

In the ASK algorithm, adaptively updating the eigenfunction approximation in the neighborhood of the solution is necessary because it is challenging to obtain very accurate approximation to the eigenfunctions, eigenvalues and Koopman modes using the initial condition only, especially for highly nonlinear systems.
When no information (e.g., range of states, regularity of the eigenfunctions) of the system is available \emph{a priori}, the adaptivity criterion serves as a updating step based on ``\emph{posterior} error estimates''. 
Furthermore, tunable parameters $r$ and $\gamma$ affect the accuracy as they are related to eigenfunction approximation and the adaptivity criterion.
Numerical analysis based on the spectrum theorem as well as the spectral method is required to systematically understand the convergence and the impact of all parameters on the performance of ASK, which will be included in our future work.

Regarding the computational complexity, as indicated in~\Cref{subsec:complexity_test}, ASK is more efficient than conventional ODE solvers when it is costly to evaluate $\bff$. 
This advantage benefits from the vectorization of evaluating function $\bff$ in modern computing tools.
Namely, ASK has the potential to outperform conventional solvers when evaluating costly function $\bff$ can be parallelized.
Nevertheless, ASK's efficiency decreases (compared with conventional solvers like Runge-Kutta) as the system dimension increases since the tensor product rule is applied to construct high-dimensional collocation points.
A possible way of improving the efficiency is to leverage the sparse grid methods to construct collocation points, which has shown its success in solving partial differential equations (PDEs) with the spectral method~\cite{shen2010efficient, shen2012efficient}.
Following this idea, we demonstrate that combined with the sparse grid method ASK can solve linear and nonlinear PDEs accurately and efficiently~\cite{li2022sparse}.
In this work, the sparse-grid-based ASK manages to solve ODEs systems (semi-discrete PDE) with dimension up to $100$. It is shown to outperform RK4 in efficiency.
Also, applying an anisotropic setting, e.g., different number of Gauss-Lobatto points, different radius, different $\gamma$ in each direction, can potentially enhance the computational efficiency.
Moreover, we provide an illustrative example on reusing computed eigenpairs of Koopman operator to solve the same ODE with new initial values. 
The advantage of ASK over conventional solvers demonstrate its potential in numerical study of the systems sensitivity, stability, uncertainty propagation, etc.

Furthermore, there are interesting relations between our work and the recently works on constructing Koopman operator's eigenfunction in an appropriate space such as~\cite{giannakis2019data, das2020koopman}.
ASK approximates eigenfunctions with orthogonal polynomials, whereas the authors use radial basis functions for the approximation in a reproducing kernel Hilbert spacein~\cite{das2020koopman}.
As an analogue, both spectral methods and radial basis methods are active topics, in the study of numerical PDEs. 
Also, in~\cite{giannakis2019data} the author uses orthogonal basis and the spectral Galerkin approach in a data-driven setting to construct eigenfunctions.
As a connection, the pseudo-spectral method can be considered as a Galerkin projection with a special measure.
Both theoretical and numerical development of the ASK method can benefit from these related studies.

Finally, since ASK is based on the Koopman operator, the spectra structure of the operator is critical in designing the algorithm such as setting parameters. 
For instance, as pointed out in~\cite{das2020koopman}, the signal will be spectrally similar to signal generated by a noisy source in the data-driven setting, if there is a non-empty continuous spectrum. Hence, it will be difficult to distinguish the true discrete spectral components. 
Also, the magnitude of the discrete spectral components carried by the signal may rapidly decay with increasing frequency.
For ASK, similar problems may lead to inaccurate approximation of the solution with a linear combination of eigenfunctions (because the continuous spectrum is associated with an integral based on an appropriate measure) or numerical issues when $N$ is large (because the magnitude of eigenvalues may decay rapidly), which requires further investigation.





\appendix
\section{An example of the obervable}
\label{append:obser}

As an example, we consider the following nonlinear dynamical system~\cite{brunton2016koopman,lusch2018deep}:
\begin{align*}
  &	\frac{\dif x_1}{\dif t} = \alpha x_1, \\
  & \frac{\dif x_2}{\dif t} = \beta (x_2 - x_1^2).
\end{align*}
Here, $\alpha$ and $\beta$ are the inherent parameters of the system. 
For such a system, appropriate observables lead to a closed-form solution. 
In particular, let $\bm y := (x_1, x_2, x_1^2)^\top$ be a three-dimensional observable. 
Then, the system can be converted to the following linear system,
\begin{align*}
    \frac{\dif \bm y}{\dif t} = 
    \begin{bmatrix}
		\alpha & 0 & 0 \\
		 0 & \beta & - \beta \\
		 0 &  0    & 2\alpha 
	\end{bmatrix}
	\boldsymbol{y}.
\end{align*}
For simplicity, assume $x_1(0) = x_2(0) = 1$. Then, we have the closed-form solution
\begin{align*}
	\boldsymbol{y} = \begin{bmatrix}
			1 \\
			0 \\
			0
		\end{bmatrix} e^{\alpha t} 
	+ 	\frac{- 2 \alpha}{\beta - 2\alpha}
		\begin{bmatrix}
			0 \\
			1 \\
			0 
		\end{bmatrix} e^{\beta t}
	+ 	\begin{bmatrix}
			0 \\
			\frac{\beta}{\beta - 2\alpha} \\
			1		
		\end{bmatrix} e^{2\alpha t}
	= 
		\begin{bmatrix}
			e^{\alpha t} \\
			\frac{-2 \alpha}{\beta - 2\alpha} e^{\beta t} + \frac{\beta}{\beta - 2\alpha} e^{2\alpha t}\\
			e^{2\alpha t}
		\end{bmatrix}. 
\end{align*}
Equivalently,
\begin{align*}
	x_1 = e^{\alpha t}, \qquad 
	x_2 = \frac{-2 \alpha}{\beta - 2\alpha} e^{\beta t} + \frac{\beta}{\beta - 2\alpha} e^{2\alpha t}.
\end{align*} 

\section{An example pseudocode}
\label{append:code}
We demonstrate a pseudo code (in MATLAB) of solving $\frac{\dif x}{\dif t} = \cos^2
(x)$, which summarizes the steps presented in ~\cref{subsec:finite_approximation}--\cref{subsec:solution}.
The MATLAB code generating Chebyshev-Gauss-Lobatto points and the associated
differentiation matrix can be found in~\cite{trefethen2000spectral}.
\begin{verbatim}
    f = @(x) cos(x).^2;  % Function f
    x0 = pi/4;           % Initial condition
    r = 0.1;             % Radius of the neighborhood (tunable)
    N = 4;               % Number of collocation points (N+1 in total)
    T = 5;               % Final time
    % Generate collocation points and the differentiation matrix 
    % on [x0-r, x0+r]
    [quad_pnt, diff_mat] = cheb(N, x0-r, x0+r);
    % Compute eigenpairs of the Koopman operator 
    K = diag(f(quad_pnt))*diff_mat;
    [eig_vec, eig_val] = eig(K, 'vector');
    % Compute coefficients (Koopman modes)
    coef = eig_vec\quad_pnt;
    % Construct solutions at time T
    sol = real(eig_vec(N/2+1,:).*coef'*exp(eig_val*T));
\end{verbatim}
When the adaptive update in ASK is activated (see~\cref{subsec:adapt}), we only need to repeat this pseudocode (as a subroutine) with an updated initial condition $x_0$ and final time $T$.


\section*{Acknowledgments}
We thank Professor Yue Yu, Hong Qian, and Yeonjoing Shin for fruitful discussions on the spectral method and properties of the Koopman operator.


\bibliographystyle{plain}
\bibliography{references}
\end{document}